%% file: main.tex
\documentclass[11pt]{article}

\usepackage{arxiv}
\usepackage{natbib}

\usepackage[utf8]{inputenc} 
\usepackage[T1]{fontenc}    
\usepackage{hyperref}       
\usepackage{url}            
\usepackage{booktabs}       
\usepackage{amsfonts}       
\usepackage{nicefrac}       
\usepackage{microtype}      
\usepackage{xcolor}         

\usepackage{algorithm}
\usepackage{algpseudocode}
\usepackage{mathtools}
\usepackage{amsmath,amsfonts,amsthm}
\usepackage{todonotes}
\usepackage{natbib}
\usepackage{makecell}
\usepackage{soul}
\usepackage{ifthen}
\usepackage{pifont}
\allowdisplaybreaks

\newcommand{\squeeze}{}

\usepackage{xspace}
\newcommand{\algname}[1]{{\sf#1}\xspace}

\input{cmd.tex}

\newcommand{\z}{\hat{z}}
\newcommand{\x}{\hat{x}}

\title{Lower Bounds and Optimal Algorithms for Smooth and Strongly Convex Decentralized Optimization Over Time-Varying Networks}

\author{%
  Dmitry Kovalev\\
  KAUST\thanks{King Abdullah University of Science and Technology, Thuwal, Saudi Arabia}\\
   \And
   Elnur Gasanov \\
   KAUST\\
   \And
   Peter Richt\'{a}rik\\
   KAUST\\
   \And
   Alexander Gasnikov\\
   MIPT\thanks{Moscow Institute of Physics and Technology, Dolgoprudny, Russia}
}

\begin{document}

\maketitle

\begin{abstract}
	We consider the task of minimizing the sum of smooth and strongly convex functions stored in a decentralized manner across the nodes of a communication network whose links are allowed to change in time. We solve two fundamental problems for this task. First, we establish {\em the first lower bounds} on the number of decentralized communication rounds and the number of local computations required to find an $\epsilon$-accurate solution. Second, we design two {\em optimal algorithms} that attain these lower bounds: (i) a variant of the recently proposed algorithm \algname{ADOM} \citep{kovalev2021adom} enhanced via a multi-consensus subroutine, which is optimal in the case when access to the dual gradients is assumed, and (ii) a novel algorithm, called \algname{ADOM+}, which is optimal in the case when access to the primal gradients is assumed. We corroborate the theoretical efficiency of these algorithms by performing an experimental comparison with existing state-of-the-art methods.	
\end{abstract}

\section{Introduction}
In this work we are solving the decentralized optimization problem
\begin{equation}\label{eq:main}
\squeeze \min\limits_{x\in\R^d} \sum\limits_{i =1}^n f_i(x),
\end{equation}
where each function $f_i \colon \R^d \rightarrow \R^d$ is stored on a compute node $i  \in \{1,\ldots, n\}$. We assume that the nodes are connected through a communication network. Each node can perform local computations based on its local state and data, and can directly communicate with its neighbors only. Further, we assume the functions $f_i$ to be smooth and strongly convex. 

Such decentralized optimization problems arise in many applications, including estimation by sensor networks \citep{rabbat2004distributed}, network resource allocation \citep{beck20141}, cooperative control \citep{giselsson2013accelerated}, distributed spectrum sensing \citep{bazerque2009distributed} and power system control \citep{gan2012optimal}. Moreover, problems of this form draw attention of the machine learning community \citep{scaman2017optimal}, since they cover training of supervised machine learning models through empirical risk minimization from the data stored across the nodes of  a network as a special case. Finally, while the current federated learning \citep{konevcny2016federated,mcmahan2017communication} systems rely on a star network topology, with a trusted server performing aggregation and coordination placed at the center of the network, advances in decentralized optimization could be useful in new-generation federated learning formulations that would rely on fully decentralized computation \citep{li2020federated}.


 
 
 \subsection{Time-varying Networks}

In this work, we focus on the practically highly relevant and theoretically challenging situation when the {\em links in the communication network are allowed to change over time}. Such {\em time-varying networks} \citep{zadeh1961time,kolar2010estimating} are ubiquitous in many complex systems and practical applications. In sensor networks, for example, changes in the link structure occur when the sensors are in motion, and due to other disturbances in the wireless signal connecting pairs of nodes. We envisage that a similar regime will be supported in future-generation federated learning systems \citep{konevcny2016federated,mcmahan2017communication}, where the communication pattern among pairs of mobile devices or mobile devices and edge servers will be dictated by their physical proximity, which naturally changes over time.

\subsection{Contributions}

In this work we present the following key contributions:
\begin{enumerate}
	\item {\bf Lower bounds.} We establish the {\em first lower bounds} on decentralized communication and local computation complexities for solving problem \eqref{eq:main} over {\em time-varying networks}. Our results are summarized in Table~\ref{tbl:complexities}, and detailed in Section~\ref{sec:lower} (see Theorems~\ref{thm:lower_comm} and~\ref{thm:lower_comp} therein).
	\item {\bf Optimal algorithms.} Further, we prove that {\em these bounds  are tight} by providing {\em two new optimal algorithms}\footnote{As a byproduct of our lower bounds, we show that the recently proposed method \algname{Acc-GT} of \citet{li2021accelerated} is also optimal. This method appeared on arXiv in April 2021, at the time when we already had a first draft of this paper, including all results. Their work does not offer any lower bounds.} which match these lower bounds: \begin{itemize}
	\item [(i)] a variant of the recently proposed algorithm \algname{ADOM} \citep{kovalev2021adom} enhanced via a multi-consensus subroutine, and 
	\item [(ii)] a novel algorithm, called \algname{ADOM+} (Algorithm~\ref{scary:alg}), also featuring multi-consensus. 
\end{itemize}	
	The former method is optimal in the case when access to the dual gradients is assumed, and the latter one is optimal in the case when access to the primal gradients is assumed. See Sections~\ref{sec:primal} and \ref{sec:optimal} for details. To the best of our knowledge, \algname{ADOM} with multi-consensus is the first dual based optimal decentralized algorithm for time-varying networks

	
	\item {\bf Experiments.} Through illustrative numerical experiments (see Section~\ref{sec:experiments}, and the extra experiments contained in the appendix) we demonstrate that our methods are {\em implementable}, and that they {\em perform competitively} when compared to existing baseline methods \algname{APM-C} \citep{rogozin2020towards, li2018sharp} and  \algname{Acc-GT} \citep{li2021accelerated}.

\end{enumerate}

\begin{table}[t]
	\caption{Current theoretical state-of-the-art methods for solving problem \eqref{eq:main} over time-varying networks, and our contributions:  lower bounds (first lower bounds for this problem), and two new optimal algorithms, \algname{ADOM} and \algname{ADOM+} with multi-consensus. }\label{tbl:complexities}	\centering
	\footnotesize
	\begin{tabular}{cccc}
		\hline
		\bf Algorithm & \makecell{\bf Local computation\\\bf  complexity }& \makecell{\bf Decentralized communication \\ \bf complexity} & \makecell{\bf Gradient \\ \bf oracle}\\
		\hline
		\multicolumn{4}{c}{\bf Known Results}\\
		\hline
		\makecell{\algname{APM-C}\\\citep{rogozin2020towards}}&
		$\cO\left(\kappa^{1/2} \log\frac{1}{\epsilon}\right)$&
		$\cO\left(\chi\kappa^{1/2} \log^{\color{blue}2}\frac{1}{\epsilon}\right)$&
		primal\\
		\hline
		\makecell{\algname{Mudag} \\ \citep{ye2020multi}}&
		$\cO\left(\kappa^{1/2} \log\frac{1}{\epsilon}\right)$&
		$\cO\left(\chi\kappa^{1/2} {\color{blue} \log(\kappa)}\log\frac{1}{\epsilon}\right)$&
		primal\\
		\hline
		\makecell{\algname{Acc-GT} with multi-consensus \\ \citep{li2021accelerated}}&
		$\cO\left(\kappa^{1/2} \log\frac{1}{\epsilon}\right)$&
		$\cO\left(\chi\kappa^{1/2}\log\frac{1}{\epsilon}\right)$&
		primal\\
		\hline
		\multicolumn{4}{c}{\bf Our Results}\\
		\hline
		\makecell{\algname{ADOM} with multi-consensus\\ This Paper, Theorem~\ref{thm:adom}}&
		$\cO\left( \kappa^{1/2} \log \frac{1}{\epsilon}\right)$&
		$\cO\left( \chi\kappa^{1/2}\log \frac{1}{\epsilon}\right)$&
		dual\\
		\hline
		\makecell{\algname{ADOM+} with multi-consensus\\ This Paper, Theorem~\ref{thm:adom+}}&
		$\cO\left( \kappa^{1/2} \log \frac{1}{\epsilon}\right)$&
		$\cO\left( \chi\kappa^{1/2}\log \frac{1}{\epsilon}\right)$&
		primal\\
		\hline
		\hline
		\makecell{\bf Lower Bounds \\ \bf This Paper, Theorems~\ref{thm:lower_comm} and~\ref{thm:lower_comp}}
		& $\cO\left( \kappa^{1/2} \log \frac{1}{\epsilon}\right)$&
		$\cO\left( \chi\kappa^{1/2}\log \frac{1}{\epsilon}\right)$&
		both
		\\
		\hline
	\end{tabular}
\end{table}

{\bf Related Work.} When the communication network is fixed in time, decentralized distributed optimization in  the strongly convex and smooth regime is relatively well studied. In particular, \citet{scaman2017optimal} established lower decentralized communication and local computation complexities for solving this problem, and proposed an optimal algorithm called \algname{MSDA} in the case when an access to the dual oracle (gradient of the Fenchel transform of the objective function) is assumed. Under a primal oracle (gradient of the objective function), current state of the art includes the near-optimal algorithms \algname{APM-C} \citep{li2018sharp, dvinskikh2019decentralized} and \algname{Mudag} \citep{ye2020multi}, and a recently proposed optimal algorithm \algname{OPAPC} \citep{kovalev2020optimal}.

The situation is worse in the time-varying case. To the best of our knowledge, no lower decentralized communication complexity bound exists for this problem. There are a few linearly-convergent  algorithms, such as those of \citet{nedic2017achieving} and \algname{Push-Pull} Gradient Method of \citet{pu2020push}, that assume a primal oracle,  and the dual oracle based algorithm \algname{PANDA} due to \citet{maros2018panda}. These algorithms have complicated theoretical analyses, which results in slow convergence rates. There are also several accelerated algorithms, which were originally developed for the fixed network case, and can be extended to the time-varying case. These include \algname{Acc-DNGD} \citep{qu2019accelerated}, \algname{Mudag} \citep{ye2020multi} and a variant of \algname{APM-C} which was extended to the time-varying case by \citet{rogozin2020towards}. Finally, there are two algorithms with state-of-the-art decentralized communication complexity: a dual based algorithm \algname{ADOM} \citep{kovalev2021adom}, and a primal based algorithm \algname{Acc-GT} \citep{li2021accelerated}.

\section{Notation and Assumptions}

\subsection{Smooth and Strongly Convex Regime}
Throughout this paper we restrict each function $f_i(x)$ to be $L$-smooth and $\mu$-strongly convex. That is, we require the following inequalities to hold for all $x,y\in \R^d$ and $i \in \{1,\ldots,n\}$:
\begin{align}
\squeeze f_i(y) + \<\g f_i(y),x-y> + \frac{\mu}{2}\sqn{x-y}	\leq f_i(x)  \leq f_i(y) + \<\g f_i(y),x-y> + \frac{L}{2}\sqn{x-y}.
\end{align}
This naturally leads to the quantity
$
	\kappa = \frac{L}{\mu} $
known as the condition number of function $f_i$.
Strong convexity implies that  problem \eqref{eq:main} has a unique solution.

\subsection{Primal and Dual Oracle}
In our work we consider two types of gradient oracles. By {\em primal oracle} we denote the situation when the gradients $\g f_i$ of the objective functions $f_i$ are available. By {\em dual oracle} we denote the situation when the gradients $\g f_i^*$ of the Fenchel conjugates\footnote{ Recall that the Fenchel conjugate of $f_i$ is given as $f_i^*(x) = \sup_{y \in \R^d} [\<x,y>  - f_i(y)]$. Note, that $f_i^*$ is $\nicefrac{1}{\mu}$-smooth and $\nicefrac{1}{L}$-strongly convex \citep{rockafellar2015convex}.} $f_i^*$ of the objective functions $f_i$ are available.

\subsection{Decentralized Communication}\label{sec:graph}
Let $\cV=\{1,\ldots, n\}$ denote the set of the compute nodes. We assume that decentralized communication is split into communication rounds. At each round $q \in \{0,1,2,\ldots\}$,  nodes  are connected through a communication network represented by a graph $\cG^q = (\cV, \cE^q)$, where $\cE^q \subset \{(i,j) \in \cV\times \cV : i\neq j\}$ is the set of links at round $q$.  For each node $i \in \cV$ we consider  a set of its immediate neighbors at round $q$: $\cN_i^q = \{j \in \cV: (i,j) \in \cE^q\}$. At round $q$, each node $i \in \cV$ can communicate with nodes from the set $\cN_i^q$ only. This type of communication is known in the literature as decentralized communication.

\subsection{Gossip Matrices} \label{sec:gossip}

Decentralized communication between nodes is typically represented via a matrix-vector multiplication with a gossip matrix. For each decentralized communication round $q \in \{0,1,2,\ldots\}$, consider a matrix $\mW(q) \in \R^{n\times n}$ with the following properties:
\begin{enumerate}
	\item $\mW(q)_{i,j} \neq 0$  if and only if $(i,j) \in \cE^q$ or $i=j$,
	\item $\ker \mW(q) \supset  \left\{ (x_1,\ldots,x_n) \in \R^n : x_1 = \ldots = x_n \right\}$,
	\item $\range \mW(q) \subset \left\{(x_1,\ldots,x_n) \in\R^n : \sum_{i=1}^n x_i = 0\right\}$,
	\item There exists $\chi \geq 1$, such that 
	\begin{equation}\label{eq:chi}
		\sqn{\mW x - x} \leq (1-\chi^{-1})\sqn{x} \text{ for all } x  \in \left\{ (x_1,\ldots,x_n) \in \R^n :\sum_{i =1}^n x_i = 0\right\}.
	\end{equation}	
\end{enumerate}

Throughout this paper we will refer to the matrix $\mW(q)$ as the {\em gossip matrix}. A typical example of a gossip matrix is $\mW(q) = \lmax^{-1}(\mL(q)) \cdot \mL(q)$, where $\mL(q)$ the Laplacian of an undirected connected graph $\cG^q$. Note that in this case $\mW(q)$ is symmetric and positive semi-definite, and $\chi$ can be chosen as an upper-bound on the condition number of the matrix $\mL(q)$ defined by $\chi = \sup_{q} \frac{\lmax(\mL(q))}{\lminp(\mL(q))}$, where $\lmax(\mL(q))$ and $\lminp(\mL(q))$ denote the largest and the smallest positive eigenvalue of $\mL(q)$, respectively. With a slight abuse of language, we will call $\chi$ the {\em condition number of time-varying network} even when the gossip matrices $\mW(q)$ are not necessarily symmetric.

\section{Lower Complexity Bounds}\label{sec:lower}

In this section we obtain lower bounds on the decentralized communication and local computation complexities for solving  problem \eqref{eq:main}. These lower bounds apply to algorithms which belong to a certain class, which we call {\em First-order Decentralized Algorithms}. Algorithms from this class need to satisfy the following assumptions:
\begin{enumerate}
	\item Each compute node can calculate first-order characteristics, such as gradient of the function it stores, or its Fenchel conjugate.
	\item Each compute node can communicate values, such as vectors from $\R^d$, with its neighbors. Note that the set of neighbors for each node is not fixed in time.
\end{enumerate}
We repeat here that when the network is fixed in time,  lower decentralized communication and local computation complexity bounds were obtained by \citet{scaman2017optimal}.

\subsection{First-order Decentralized Algorithms}\label{sec:alg_class}

Here we give a formal definition of the class of algorithms for which we provide lower bounds on the  complexity of solving problem \eqref{eq:main}.  At each time step $k \in \{0,1,2,\ldots\}$, each node $i \in \cV$ maintains a finite local memory $\cH_i(k) \subset \R^d$.  For simplicity, we assume that the local memory is initialized as $\cH_i(k) = \{0\}$.

At each time step $k$, an algorithm  either performs a decentralized communication round, or a local computation round that can update the memory $\cH_i(k)$.  The update of the local memory satisfies the following rules:
\begin{enumerate}
	\item 
	If an algorithm performs a local computation round at time step $k$, then
	\begin{equation*}
	\cH_i(k+1) \subset \Span(\{x, \g f_i(x), \g f_i^*(x) : x \in \cH_i(k)\})
	\end{equation*}
	for all $i \in \cV$, where $f_i^*$ is the Fenchel conjugate of $f_i$.
	\item
	If an algorithm performs a decentralized communication round at time step $k$, then
	\begin{equation*}
	\cH_i(k+1) \subset \Span\left(\cup_{j \in \cN_i^{q} \cup \{i\}} \cH_j(k)\right)
	\end{equation*}
	for all $i \in \cV$, where $q \in \{0,1,2,\ldots\}$ refers to the number of  decentralized communication round performed so far, counting from the first one, for which we set $q=0$. See Figure~\ref{fig:iterations} for a graphical illustration of our time step ($k$) and decentralized communication ($q$) counters.
\end{enumerate}
At time step $k$, each node $i \in \cV$ must specify an output value $x_i^o(k) \in \cH_i(k)$. 

\begin{figure}
	\centering
\begin{tikzpicture}
\def\array{{0,1,1,0,0,1,1,1,0}};
\edef\q{0}
\draw[black, thick,->] (-0.3,0) -- (4.7,0);
\draw (4.7,0) node[anchor=north] {$k$};

\draw[black, thick,->] (-0.3,0.5) -- (4.7,0.5);
\draw (4.7,0.5) node[anchor=north] {$q$};

\foreach \i in {0,...,8}
{
	\draw[black, thick] (\i / 2,-0.07) -- (\i/2,0.07);
	\draw (\i / 2,0)node[anchor=north] {$\i$};
	\pgfmathsetmacro{\x}{\array[\i]}
	\draw (\i / 2,1) node {\ifthenelse{\x=1}{\ding{51}}{\ding{55}}};
	\draw (\i / 2,1.5) node {\ifthenelse{\x=1}{\ding{55}}{\ding{51}}};
	
	\draw[black, thick] (\i / 2,-0.07+0.5) -- (\i/2,0.07+0.5);
	\ifthenelse{\x=1}{\draw (\i / 2,0.5)node[anchor=north] {$\q$};}{};
	
	\pgfmathparse{int(\q+\array[\i])}
	\xdef\q{\pgfmathresult}
}

\draw(-0.3,0) node[anchor=east]{time step $k$};
\draw(-0.3,0.5) node[anchor=east]{communication round number $q$};
\draw(-0.3,1) node[anchor=east]{decentralized communication};
\draw(-0.3,1.5) node[anchor=east]{local computation};
\end{tikzpicture}
\caption{The way we count  decentralized communication rounds ($q$)  and iterations ($k$).}\label{fig:iterations}
\end{figure}
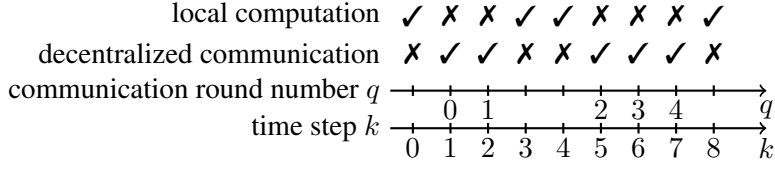

In the case of time-invariant networks, \citet{scaman2017optimal} presented lower complexity bounds for a class of algorithms called {\em Black-box Optimization Procedures}. This class is slightly more general than our  First-order Decentralized Algorithms. 
In particular, we use discrete time $k \in \{0,1,2,\ldots\}$ rather than continuous time, and do not allow decentralized communications and local computations to be  performed in parallel and asynchronously.
We have done this for the sake of simplicity and clarity. This should not be seen as a weakness of our work, since our results can be easily extended to Black-box Optimization Procedures. However, we are of the opinion that such an extension would not give any new substantial  insights.

\subsection{Main Theorems}

We are now ready to present our main theorems describing the first lower bounds on the number of decentralized communication rounds and local computation rounds that are necessary to find an approximate solution of the problem \eqref{eq:main}.
In order to simplify the proofs, and following the approach used by \citet{nesterov2003introductory}, we consider the limiting case $d \rightarrow +\infty$. More precisely, we will work in the infinite dimensional space 
$
\squeeze \ell_2 = \left\{x = (x_{[l]})_{l=1}^\infty  : \sum_{l=1}^\infty (x_{[l]})^2 < +\infty\right\}
$
rather than $\R^d$.

\begin{theorem}\label{lower:thm:comm}
	Let $\chi \geq 3$, $L > \mu \geq 0$. There exists a sequence of graphs $\{\cG^q\}_{q=0}^\infty$, which satisfy assumptions from Section~\ref{sec:graph},  a corresponding sequence of gossip matrices $\{\mW(q)\}_{q=0}^\infty \subset \R^{n\times n}$ satisfying assumptions from Section~\ref{sec:gossip}, and $\mu$-strongly convex and $L$-smooth functions $f_i: \ell_2 \rightarrow \R, i \in \cV$, such that for any first-order decentralized algorithm and any $k \in \{0,1,2\ldots\}$
	\begin{equation}
	\squeeze	\sqn{x_i^o(k) - x^*} \geq C \left(\max\left\{0,1 - \frac{24\sqrt{6\mu}}{\sqrt{L}} \right\}\right)^{q/\chi},
	\end{equation}
	where $q$ is the number of decentralized communication rounds performed by the algorithm before time step $k$, $x^*$ is the solution of the problem \eqref{eq:main} and $C > 0$ is some constant.
\end{theorem}

Our proof is inspired by the proof of \citet{scaman2017optimal} for time-invariant networks, which in its turn is based on the proof of oracle complexities for strongly convex and smooth optimization by \citet{nesterov2003introductory}. Here we give a short outline of our proof.

We choose  $n = |\cV| \approx \chi$ nodes and split them into three disjoint sets, $\cV_1, \cV_2$ and $\cV_3$, of equal size $\nicefrac{n}{3}$. Further, we split the function used by \citet{nesterov2003introductory} onto the nodes belonging to $\cV_1$ and $\cV_3$. One then needs to show that most dimensions of the output vectors $x_i^o(k)$ will remain zero, while local computations may only increase the number of non-zero dimensions by one. In contrast to \citep{scaman2017optimal}, we need to have $|\cV_2| \approx \chi$ rather than $|\cV_2| \approx \sqrt{\chi}$, and still ensure that at least $|\cV_2|$ decentralized communication rounds are necessary to share information between node groups $\cV_1$ and $\cV_3$. We achieve this by choosing a sequence of star graphs with the center node cycling through the nodes from $\cV_2$. 
\begin{theorem}\label{thm:lower_comm}
	For any $\chi \geq 3$, $L > \mu \geq 0$ there exists a sequence of  graphs $\{\cG^q\}_{q=0}^\infty$ satisfying assumptions from Section~\ref{sec:graph},  a corresponding sequence of gossip matrices $\{\mW(q)\}_{q=0}^\infty \subset \R^{n\times n}$ satisfying assumptions from Section~\ref{sec:gossip}, and $\mu$-strongly convex and $L$-smooth functions $f_i: \ell_2 \rightarrow \R, i \in \cV$, such that for any first order decentralized algorithm, the number of decentralized communication rounds to find an $\epsilon$-accurate solution of the problem \eqref{eq:main} is lower bounded by
	\begin{equation}
	\squeeze	\Omega \left( \chi\sqrt{\nicefrac{L}{\mu}}\log \frac{1}{\epsilon}\right).
	\end{equation}
\end{theorem}


We also provide a lower bound on the local computation complexity. The proof is similar to the proof of Theorem~\ref{lower:thm:comm}.

\begin{theorem}\label{thm:lower_comp}
	For any $\chi \geq 3$, $L > \mu \geq 0$ there exists a sequence of  graphs $\{\cG^q\}_{q=0}^\infty$ satisfying assumptions from Section~\ref{sec:graph},  a corresponding sequence of gossip matrices $\{\mW(q)\}_{q=0}^\infty \subset \R^{n\times n}$ satisfying assumptions from Section~\ref{sec:gossip}, and $\mu$-strongly convex and $L$-smooth functions $f_i: \ell_2 \rightarrow \R, i \in \cV$, such that for any first order decentralized algorithm, the number of local computation rounds to find an $\epsilon$-accurate solution of the problem \eqref{eq:main} is lower bounded by
	\begin{equation}
\squeeze	\Omega \left( \sqrt{\nicefrac{L}{\mu}}\log \frac{1}{\epsilon}\right).
	\end{equation}
\end{theorem}

The detailed proofs are available in the appendix.

\section{Primal Algorithm with Optimal Communication Complexity} \label{sec:primal}


In this section we develop a novel algorithm for decentralized optimization over time-varying networks with optimal decentralized communication complexity. This is a primal algorithm, meaning that it uses the primal gradient oracle. The design of this algorithm relies on a sequence of specific reformulations of the problem \eqref{eq:main}, which we now  describe.

\subsection{Reformulation via Lifting}
Consider a function $F \colon (\R^d)^\cV \rightarrow \R$ defined by
\begin{equation}
\squeeze	F(x) = \sum \limits_{i \in \cV} f_i(x_i),
\end{equation}
where $x = (x_1,\ldots,x_n) \in (\R^d)^\cV$. This function is $L$-smooth and $\mu$-strongly convex since the individual functions $f_i$ are. Consider also the so called consensus space $\cL \subset(\R^d)^\cV$ defined by 
\begin{equation}\label{eq:L}
	\cL = \{(x_1,\ldots,x_n) \in (\R^d)^\cV: x_1 = \ldots = x_n\}.
\end{equation}
Using this notation, we arrive at the equivalent formulation
of problem \eqref{eq:main}
\begin{equation}\label{eq:primal}
	\min_{x \in \cL} F(x).
\end{equation}
Due to strong convexity, this reformulation has a unique solution, which we denote as $x^* \in \cL$.

\subsection{Saddle Point Reformulation}

Next, we introduce a reformulation of problem \eqref{eq:primal} using a parameter $\nu \in (0,\mu)$ and a slack variable $w \in (\R^d)^\cV$:
\begin{equation*}
\squeeze	\min\limits_{\substack{x, w \in (\R^d)^\cV\\w=x, w \in \cL}} F(x) - \frac{\nu}{2}\sqn{x} + \frac{\nu}{2}\sqn{w}.
\end{equation*}
Note that the function $F(x) -\frac{\nu}{2}\sqn{x}$ is $(\mu - \nu)$-strongly convex since $\nu < \mu$. The latter problem is a minimization problem with linear constraints. Hence, it has the equivalent saddle-point reformulation
\begin{equation*}
\squeeze	\min \limits_{x,w \in (\R^d)^\cV}\max\limits_{y\in (\R^d)^\cV}\max\limits_{z \in \cL^\perp} F(x) - \frac{\nu}{2}\sqn{x} + \frac{\nu}{2}\sqn{w} + \<y,w-x> + \<z,w>,
\end{equation*}
where $\cL^\perp \subset  (\R^d)^\cV$ is an orthogonal complement to the space $\cL$, defined by
\begin{equation}\label{eq:Lperp}
	\cL^\perp = \left\{(z_1\ldots,z_n) \in (\R^d)^\cV : {\sum_{i =1}^n} z_i = 0\right\}.
\end{equation}
Minimization in $w$ gives the final saddle-point reformulation of the problem \eqref{eq:primal}:
\begin{equation}\label{eq:saddle}
\squeeze	\min \limits_{x\in (\R^d)^\cV}\max\limits_{y\in (\R^d)^\cV}\max\limits_{z \in \cL^\perp} F(x) - \frac{\nu}{2}\sqn{x}  - \<y,x>  - \frac{1}{2\nu}\sqn{y+z}.
\end{equation}
Further, by $\sE$ we denote the Euclidean space $\sE = (\R^d)^\cV\times (\R^d)^\cV \times \cL^\perp$.
One can show that the saddle-point problem \eqref{eq:saddle} has a unique solution $(x^*,y^*,z^*) \in \sE$, which satisfies the following optimality conditions:
\begin{align}
0&= \g F(x^*)  - \nu x^* - y^*,\label{opt:x}\\
0&= \nu^{-1}(y^* + z^*) + x^*,\label{opt:y}\\
\cL &\ni y^* + z^*.\label{opt:z}
\end{align}

\subsection{Monotone Inclusion Reformulation}
 Consider two monotone operators $A,B \colon \sE \rightarrow \sE$, defined via
\begin{equation}\label{eq:AB}
\squeeze  A(x,y,z) = \vect{\g F(x) - \nu x\\ 
\squeeze \nu^{-1}(y+z)\\\mP\nu^{-1}(y+z)},
\squeeze  \qquad B(x,y,z) = \vect{-y\\x\\0},
\end{equation}
where $\mP$ is an orthogonal projection matrix onto the subspace $\cL^\perp$. Matrix $\mP$ is given as \begin{equation}
\mP = (\mI_n - \tfrac{1}{n}\ones_n\ones_n^\top) \otimes \mI_d,
\end{equation}
where $\mI_p$ denotes $p\times p$ identity matrix, $\ones_n = (1,\ldots,1) \in \R^n$, and $\otimes$ is the Kronecker product.  Then, solving problem \eqref{eq:saddle} is equivalent to finding $(x^*,y^*, z^*) \in \sE$, such that
\begin{equation}\label{eq:AB0}
	A(x^*,y^*,z^*) + B(x^*,y^*,z^*) = 0.
\end{equation}
Indeed, optimality condition \eqref{opt:z} is equivalent to $\proj_{\cL^\perp}(y^* + z^*) = 0$ or $\mP\nu^{-1}(y^* + z^*)  = 0$. Now, it is clear that \eqref{eq:AB0} is just another way to write the optimality conditions for  problem \eqref{eq:saddle}.

\subsection{Primal Algorithm Design and Convergence}
A common approach to solving  problem \eqref{eq:AB0} is to use the \algname{Forward-Backward} algorithm. The update rule of this algorithm is $
	(x^+,y^+,z^+) = J_{\omega B}[(x,y,z) - \omega A(x,y,z)],
$
where $\omega >0$ is a stepsize, the operator $J_{\omega B}\colon \sE \rightarrow \sE$ is called the resolvent, and is defined as the inverse of the  operator $(I + \omega B)\colon \sE \rightarrow \sE$, where $I \colon \sE \rightarrow \sE$ is the identity mapping. One can observe that the resolvent  $J_{\omega B}$ is easy to compute.

Following \citep{kovalev2020optimal}, we use an accelerated version of  the \algname{Forward-Backward} algorithm. This can indeed be done, since the operator $A\colon \sE \rightarrow \sE$ is the gradient of the smooth and convex function $(x,y,z) \mapsto F(x)  - \frac{\nu}{2}\sqn{x}+ \frac{1}{2\nu}\sqn{y+z}$ on the Euclidean space $\sE$.  Note that the operator $A+B$ is {\em not} strongly monotone, while strong monotonicity is usually required to achieve linear convergence. However, we can still obtain a linear convergence rate by carefully utilizing the $L$-smoothness property of the function $F(x)$. Similar issue appeared  in the design of algorithms for solving linearly-constrained minimizations problems, including the non-accelerated algorithms \citep{condat2019proximal, salim2020dualize} and the optimal algorithm \citep{salim2021optimal}. However, the authors of these works considered a different  problem reformulation from \eqref{eq:saddle}, and hence their results can not be applied here.

However, one issue still remains: to compute the operator $A$, it is necessary to perform a matrix-vector multiplication with the matrix  $\mP$, which requires full averaging, i.e, consensus over all nodes of the network. Following the approach of \citet{kovalev2021adom}, we replace it with the multiplication via the gossip matrix $\mW(q) \otimes \mI_d$, which requires one decentralized communication round only. That is, we replace the last component  $\mP\nu^{-1}(y+z)$ of the operator $A$ defined by \eqref{eq:AB} with $(\mW(q) \otimes \mI_d)\nu^{-1}(y+z)$. \citet{kovalev2021adom} showed that multiplication with the gossip matrix can be seen as a compression on the Euclidean space $\cL^\perp$, i.e., condition \eqref{eq:chi} holds, and hence the so-called error-feedback mechanism \citep{stich2019error,karimireddy2019error,gorbunov2020linearly} can be applied to obtain linear convergence.  We use this insight in the design of our algorithm.

Armed with all these ideas, we are ready to present our method \algname{ADOM+}; see Algorithm~\ref{scary:alg}.
\begin{algorithm}[H]
	\caption{\algname{ADOM+}}
	\label{scary:alg}
	\begin{algorithmic}[1]
		\State {\bf input:} $x^0, y^0,m^0 \in (\R^d)^\cV$, $z^0 \in \cL^\perp$
		\State $x_f^0 = x^0$, $y_f^0 = y^0$, $z_f^0 = z^0$
		\For{$k = 0,1,2,\ldots$}
		\State $x_g^k = \tau_1x^k + (1-\tau_1) x_f^k$\label{scary:line:x:1}
		\State $x^{k+1} = x^k + \eta\alpha(x_g^k - x^{k+1}) - \eta\left[\g F(x_g^k) - \nu x_g^k - y^{k+1}\right] $\label{scary:line:x:2}
		\State $x_f^{k+1} = x_g^k + \tau_2 (x^{k+1} - x^k)$\label{scary:line:x:3}
		
		\State $y_g^k = \sigma_1 y^k + (1-\sigma_1)y_f^k$\label{scary:line:y:1}
		\State $y^{k+1} = y^k + \theta\beta (\g F(x_g^k) - \nu x_g^k - y^{k+1}) -\theta\left[\nu^{-1}(y_g^k + z_g^k) + x^{k+1}\right]$\label{scary:line:y:2}
		\State $y_f^{k+1} = y_g^k + \sigma_2 (y^{k+1} - y^k)$\label{scary:line:y:3}
		
		\State $z_g^k = \sigma_1 z^k + (1-\sigma_1)z_f^k$\label{scary:line:z:1}
		\State $z^{k+1} = z^k + \gamma \delta(z_g^k - z^k) - (\mW(k)\otimes \mI_d)\left[\gamma\nu^{-1}(y_g^k+z_g^k) + m^k\right]$\label{scary:line:z:2}
		\State $m^{k+1} = \gamma\nu^{-1}(y_g^k+z_g^k) + m^k - (\mW(k)\otimes \mI_d)\left[\gamma\nu^{-1}(y_g^k+z_g^k) + m^k\right]$\label{scary:line:m}
		\State $z_f^{k+1} = z_g^k - \zeta (\mW(k)\otimes \mI_d)(y_g^k + z_g^k)$\label{scary:line:z:3}
		\EndFor
	\end{algorithmic}
\end{algorithm}

We now establish the convergence rate of Algorithm~\ref{scary:alg}.

\begin{theorem}[Convergence of \algname{ADOM+}]\label{thm:adom+_conv}
	To reach precision $\sqn{x^k - x^*} \leq \epsilon$, Algorithm~\ref{scary:alg} requires  the following number of iterations
	\begin{equation*}
	\squeeze	\cO \left( \chi\sqrt{\nicefrac{L}{\mu}}\log \frac{1}{\epsilon}\right).
	\end{equation*}
\end{theorem}

Note that Algorithm~\ref{scary:alg} performs $\cO(1)$ decentralized communication and local computation rounds. Hence, its decentralized communication  and local computation complexities are $\cO \left( \chi\sqrt{\nicefrac{L}{\mu}}\log \frac{1}{\epsilon}\right)$.

\section{Optimal Algorithms} \label{sec:optimal}

In this section we develop decentralized algorithms for time-varying networks with optimal local computation and decentralized communication complexities. We develop both primal and dual algorithms that use primal and dual gradient oracle, respectively. The key mechanism to obtain optimal algorithms is to incorporate the multi-consensus procedure into the algorithms with optimal decentralized communication complexity.

\subsection{Multi-consensus Procedure}

As discussed in Section~\ref{sec:gossip}, a decentralized communication round cab be represented as the multiplication with the gossip matrix $\mW(q)$. The main idea behind the multi-consensus procedure is to replace the matrix $\mW(q)$ with another matrix, namely
\begin{equation}\label{eq:W_multi}
\squeeze	\mW(k; T) = \mI - \prod \limits_{q = kT}^{(k+1)T -1} (\mI - \mW(q)),
\end{equation}
where $k \in \{0,1,2,\ldots\}$ is the iteration counter, and $T \in \{1,2,\ldots\}$ is the number of consensus steps. One can show that this matrix satisfies the assumptions on the gossip matrix from Section~\ref{sec:gossip}, including the contraction property \eqref{eq:chi}:
\begin{equation}\label{eq:W_multi_contraction}
	\sqn{\mW(k;T) x - x} \leq (1-\chi^{-1})^T\sqn{x} \text{ for all } x \in \cL^\perp.
\end{equation}
One can also observe that multiplication with the matrix $\mW(k;T)$ requires to perform multiplication with $T$ gossip matrices $\mW(kT) , \mW(kT+1), \ldots, \mW(kT + T - 1)$, hence it requires $T$ decentralized communication rounds.

\subsection{Optimal Algorithms: ADOM and ADOM+ with multi-consensus}

Now, we are ready to describe our optimal algorithms for smooth and strongly decentralized optimization over time-varying networks. As mentioned before, we start with algorithms with optimal decentralized communication complexity. In the case when access to the primal oracle is assumed, \algname{ADOM+} (Algorithm~\ref{scary:alg}) is the algorithm of choice. In the case when access to the dual oracle is assumed, \citet{kovalev2021adom} proposed the dual based accelerated decentralized algorithm \algname{ADOM}. The original convergence proof of this algorithm requires the gossip matrix $\mW(k)$ to be symmetric. However, the generalization of this proof to the case when the gossip matrix satisfies assumptions from Section~\ref{sec:gossip}, and is not necessarily symmetric, is straightforward.

Both \algname{ADOM} and \algname{ADOM+} require $\cO\left(\chi\sqrt{\nicefrac{L}{\mu}}\log\frac{1}{\epsilon}\right)$ iterations to obtain an $\epsilon$-accurate solution of problem \eqref{eq:main}, and at each iteration they perform $\cO(1)$ local computation and decentralized communication rounds, i.e., multiplications with the gossip matrix $\mW(q)$. We now incorporate a multi-consensus procedure into both algorithms. That is, we replace matrix $\mW(q)$ with the matrix $\mW(k; T)$ defined in \eqref{eq:W_multi}, where $k \in \{0,1,2,\ldots\}$ is the iteration counter. As mentioned before, this is equivalent to performing $\cO(T)$ decentralized communication rounds at each iteration. Choosing the number of consensus steps $T = \ceil{\chi \ln2}$ together with \eqref{eq:W_multi_contraction} implies
\begin{equation}
\squeeze \sqn{\mW(k;T)x - x} \leq \frac{1}{2}\sqn{x}\text{ for all } x \in \cL^\perp.
\end{equation}
This means that $\mW(k;T)$ satisfies \eqref{eq:chi} with $\chi$ replaced by $\frac{1}{2}$ and hence both \algname{ADOM} and \algname{ADOM+} with multi-consensus require $\cO\left(\sqrt{\nicefrac{L}{\mu}}\log\frac{1}{\epsilon}\right)$ iterations to obtain an $\epsilon$-accurate solution.
Taking into account that these algorithms still perform $\cO(1)$ local computations and $\cO(T) = \cO(\chi)$ decentralized communication rounds at each iteration, we arrive at the following theorems.

\begin{theorem}\label{thm:adom}
	\algname{ADOM} with multi-consensus requires $\cO\left(\sqrt{\nicefrac{L}{\mu}}\log\frac{1}{\epsilon}\right)$ local computation rounds and $\cO\left(\chi\sqrt{\nicefrac{L}{\mu}}\log\frac{1}{\epsilon}\right)$  decentralized communication rounds to find an $\epsilon$-accurate solution of the distributed optimization  problem \eqref{eq:main}.
\end{theorem}

\begin{theorem}\label{thm:adom+}
	\algname{ADOM+} with multi-consensus requires $\cO\left(\sqrt{\nicefrac{L}{\mu}}\log\frac{1}{\epsilon}\right)$ local computation rounds and $\cO\left(\chi\sqrt{\nicefrac{L}{\mu}}\log\frac{1}{\epsilon}\right)$  decentralized communication rounds to find an $\epsilon$-accurate solution of the distributed optimization  problem \eqref{eq:main}.
\end{theorem}

\subsection{Comparison with Acc-GT \citep{li2021accelerated}}
To the best of our knowledge \algname{ADOM} with multi-consensus is the first dual based optimal decentralized algorithm for time-varying networks. In the case when the primal oracle is assumed, besides \algname{ADOM+} with multi-consensus, there is only one previously existing algorithm, \algname{Acc-GT} with multi-consensus \citep{li2021accelerated}, which achieves optimal local computation and decentralized communication complexities. The main advantage of \algname{Acc-GT} over \algname{ADOM+} is that while its convergence theory supports so-called $\gamma$-connected networks, we do not consider this case in our work. However, in the case when multi-consensus is not used, its iteration complexity is $\cO\left(\chi^{3/2}\sqrt{\nicefrac{L}{\mu}}\log\frac{1}{\epsilon}\right)$, which is {\em worse} than the complexity of \algname{ADOM+}, by the factor $\chi^{1/2}$. Since communication is known to be the main bottleneck in distributed training systems, multi-consensus is unlikely to be used in practice, which may limit the practical performance of  \algname{Acc-GT} in comparison to \algname{ADOM+}.

\section{Experiments} \label{sec:experiments}

\begin{figure}[t]
	\includegraphics[width=0.24\textwidth]{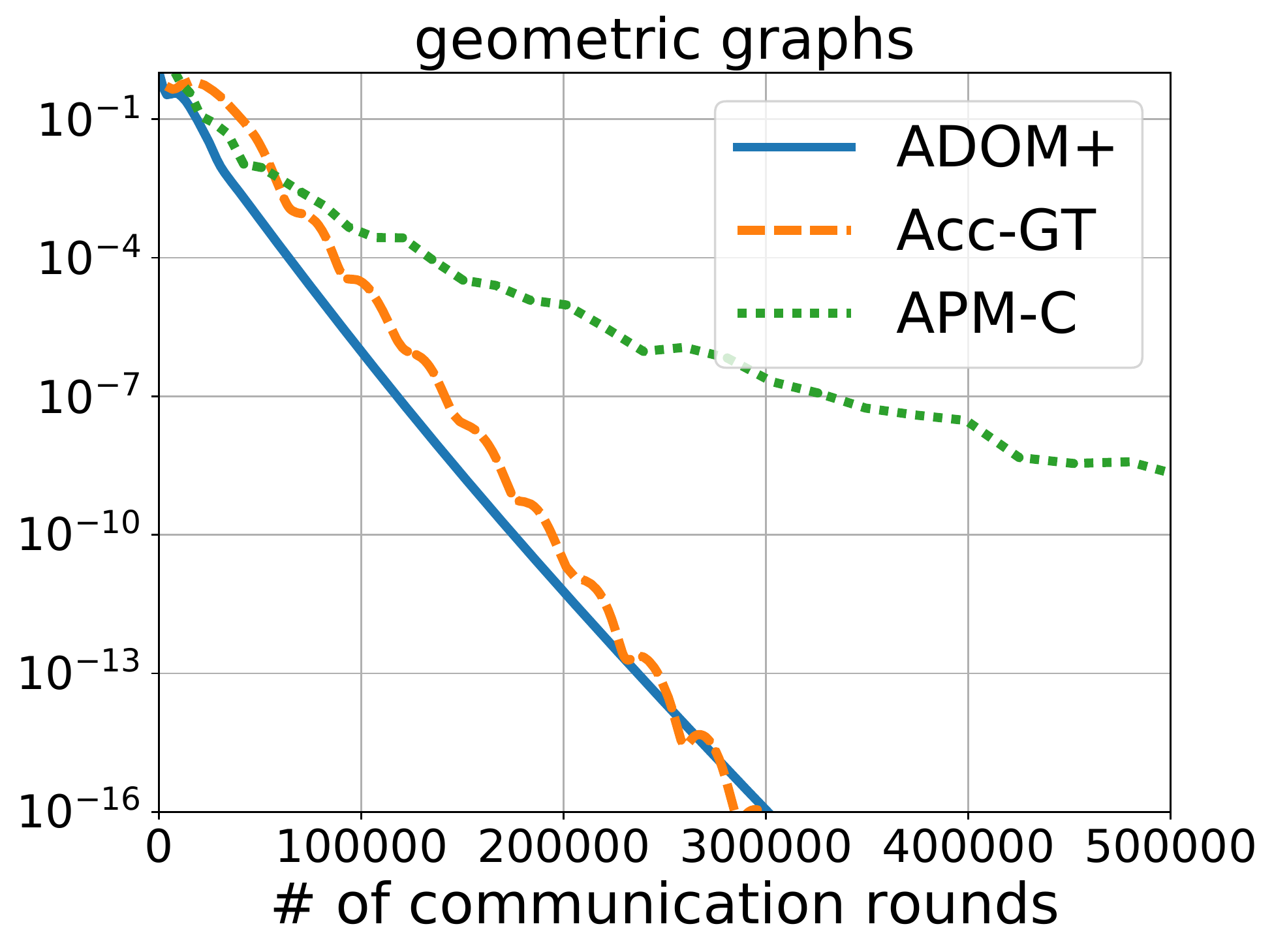}
	\includegraphics[width=0.24\textwidth]{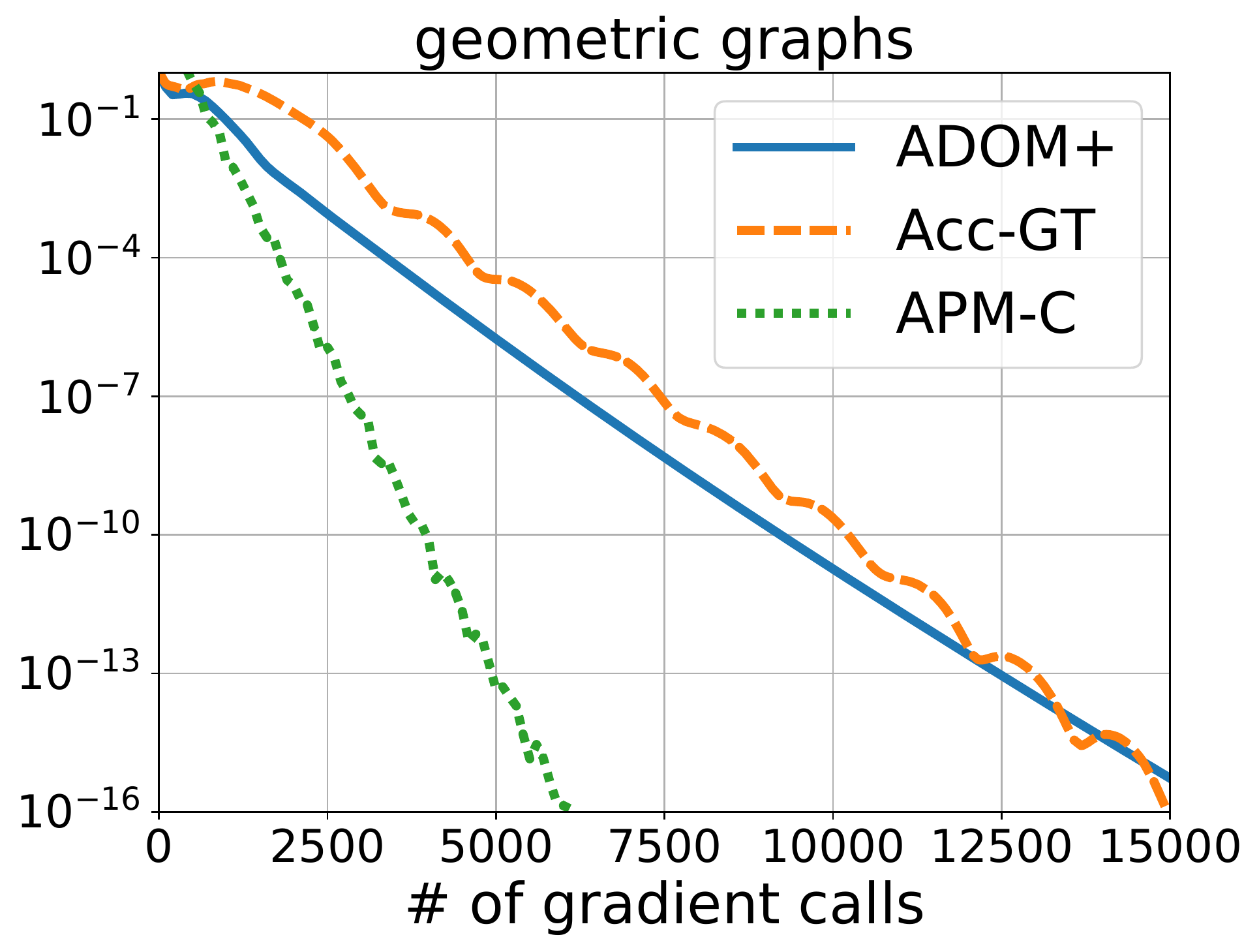}
	\includegraphics[width=0.24\textwidth]{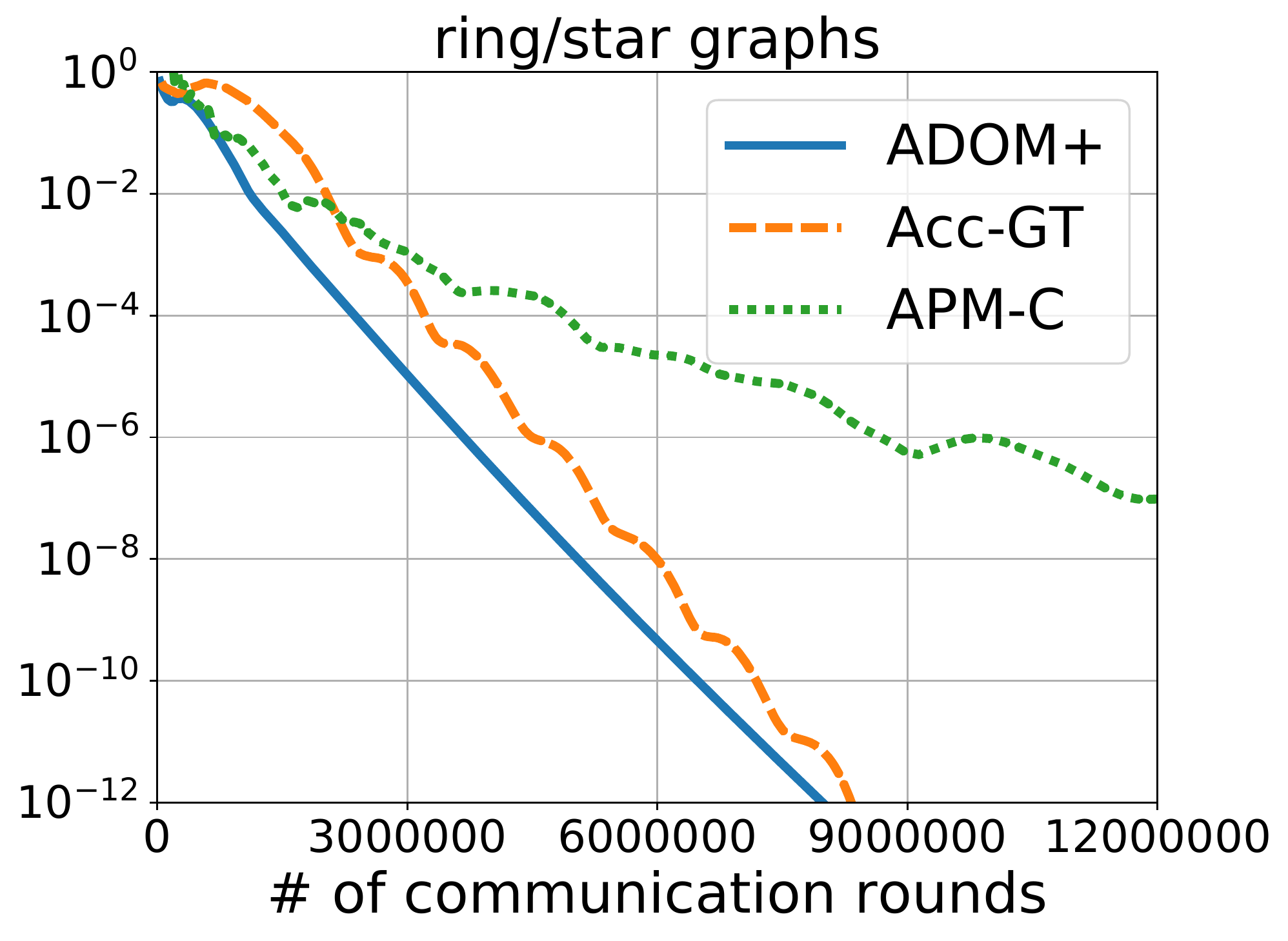}
	\includegraphics[width=0.24\textwidth]{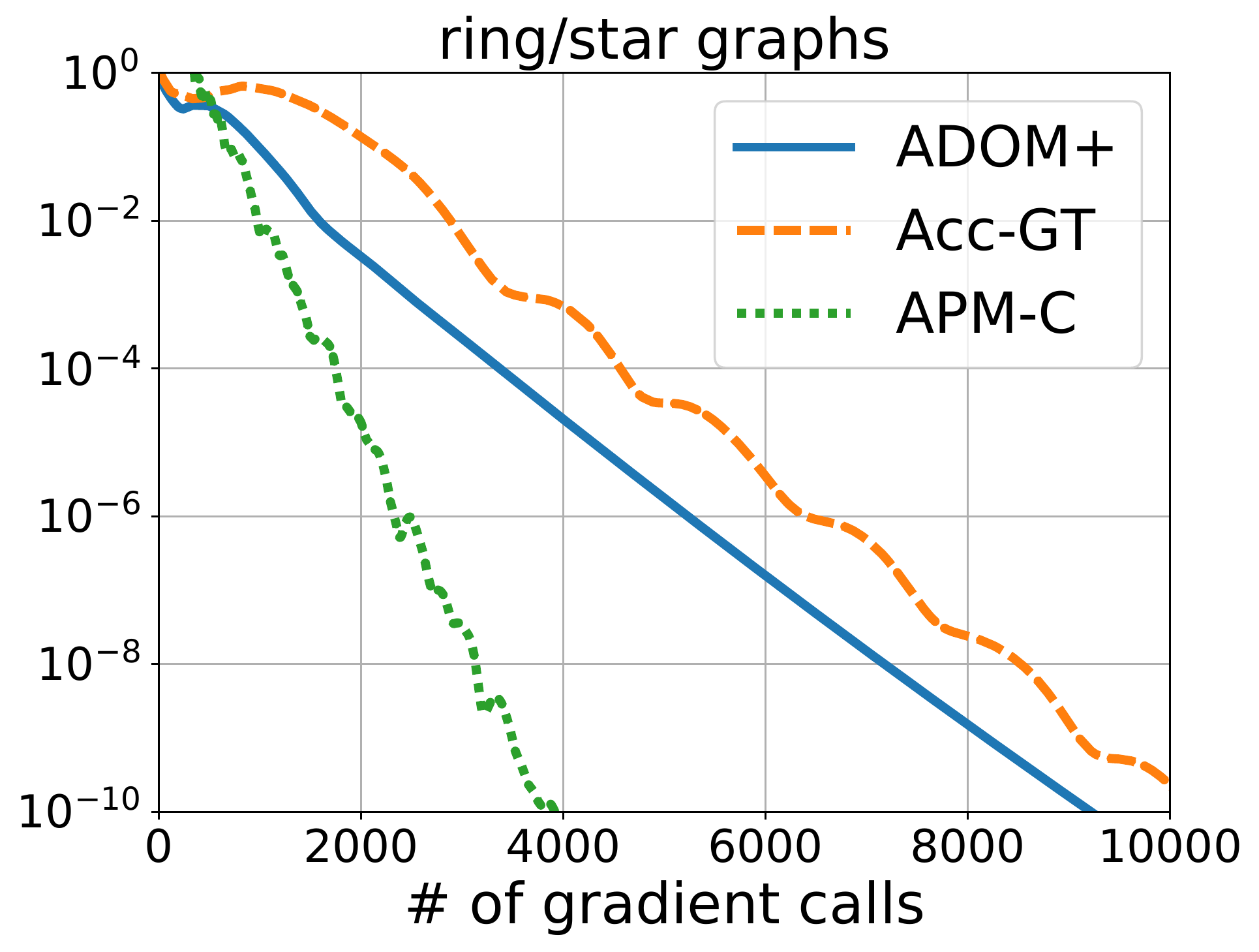}
	
	\caption{Comparison of our method \algname{ADOM+} and the baselines  \algname{Acc-GT} and \algname{APM-C}.}
	\label{fig:covtype}
\end{figure}

In this section we perform an illustrative experiment with logistic regression. We take $10,000$ samples from the {\tt covtype} LIBSVM\footnote{The LIBSVM~\citep{chang2011libsvm} dataset collection is available at\\ \href{https://www.csie.ntu.edu.tw/~cjlin/libsvmtools/datasets/}{https://www.csie.ntu.edu.tw/~cjlin/libsvmtools/datasets/}} dataset and distribute them across $n=100$ nodes of a network, $100$ samples per node. We use two types of networks: a sequence of random geometric graphs with $\chi \approx 30$, and a sequence which alternates between the ring and star topology, with $\chi \approx 1,000$. We choose the regularization parameter $\mu$ such that the condition number becomes $\kappa = 10^5$. We compare \algname{ADOM+} with multi-consensus with two state-of-the art decentralized algorithms  for time-varying networks: \algname{Acc-GT} with multi-consensus \citep{li2021accelerated}, and a variant of \algname{APM-C} for time-varying networks \citep{rogozin2020towards, li2018sharp}. We set all  parameters of \algname{Acc-GT} and \algname{APM-C} to those used in the experimental sections of the corresponding papers, and tune the parameters of \algname{ADOM+}. The results are presented in Figure~\ref{fig:covtype}. We see that \algname{ADOM+} has  similar empirical behavior to the recently proposed \algname{Acc-GT} method. Both these methods are better than \algname{APM-C} in terms of the number of decentralized communication rounds (this is expected, since this method has sublinear communication complexity), and worse in terms of the number of gradient calls. We provide more details and additional experiments in the appendix.

\bibliographystyle{apalike}
\bibliography{ref.bib}

\clearpage


\appendix

\newpage

\input{appendix.tex}

\end{document}

%% file: cmd.tex
\DeclarePairedDelimiter\ceil{\lceil}{\rceil}
\DeclarePairedDelimiter\floor{\lfloor}{\rfloor}

\newtheorem{theorem}{Theorem}
\newtheorem{lemma}{Lemma}

\newcommand{\proj}{\mathrm{proj}}
\newcommand{\range}{\mathrm{range}}

\newcommand{\Span}{\mathrm{Span}}

\newcommand{\lmax}{\lambda_{\max}}
\newcommand{\lmin}{\lambda_{\min}}
\newcommand{\lminp}{\lambda_{\min}^+}

\newcommand{\g}{\nabla}
\newcommand{\ones}{\mathbf{1}}

\newcommand{\bg}{\mathrm{D}}

\newcommand{\R}{\mathbb{R}}

\def\<#1,#2>{\langle #1,#2\rangle}

\newcommand{\norm}[1]{\|#1\|}
\newcommand{\sqn}[1]{\norm{#1}^2}

\newcommand{\vect}[1]{\begin{bmatrix*}[c]#1\end{bmatrix*}}

\newcommand{\sE}{\mathsf{E}}

\newcommand{\cH}{\mathcal{H}}

\newcommand{\cN}{\mathcal{N}}

\newcommand{\cG}{\mathcal{G}}

\newcommand{\cV}{\mathcal{V}}
\newcommand{\cE}{\mathcal{E}}
\newcommand{\cL}{\mathcal{L}}

\newcommand{\cO}{\mathcal{O}}

\newcommand{\mW}{\mathbf{W}}

\newcommand{\mL}{\mathbf{L}}
\newcommand{\mI}{\mathbf{I}}
\newcommand{\mP}{\mathbf{P}}

\newcommand{\eqdef}{\coloneqq}

%% file: appendix.tex
\part*{Appendix}


\section{Experimental Details and Additional Experiments}

\subsection{Experimental Details}
We perform experiments with logistic regression for binary classification with $\ell_2$ regularization. That is, our loss function has the form
\begin{equation}
	f_i(x) = \frac{1}{m}\sum_{j=1}^m \log(1 + \exp(-b_{ij}a_{ij}^\top x)) + \frac{r}{2}\sqn{x},
\end{equation}
where $a_{ij} \in \R^d$ and $b_{ij} \in \{-1,+1\}$ are data points and labels, $r>0$ is a regularization parameter, and $m$ is the number of data points stored on each node. In this section we generate synthetic datasets with \texttt{sklearn.datasets.make classification} function from scikit-learn library. We generate a number of datasets consisting of $10,000$ samples, distributed across $n=100$ nodes of the network, $m=100$ samples per each node. We vary the parameter $r$ to obtain different values of the condition number $\kappa$.

\subsection{Further experiments}

Here we simulate time-varying networks with a sequence of randomly generated geometric graphs. Geometric graphs are constructed by generating $n=100$ nodes from the uniform distribution over $[0,1]^2 \subset \R^2$, and connecting each pair of nodes whose distance is less than a certain {\em radius}. We  enforce connectivity by adding a small number of edges. We obtain a sequence of graphs $\{\cG^q\}_{q=0}^\infty$ by generating a set of $50$ random geometric graphs, and alternating between them in a cyclic manner.
We choose $\mW(q)$ to be the Laplacian matrix of the graphs $\cG^q$ divided by its largest eigenvalue. We vary the condition number $\chi$ by choosing different values of the radius parameter.

We compare \algname{ADOM+} with state-of-the art primal decentralized algorithms for time-varying networks: \algname{Acc-GT} \citep{li2021accelerated} and a variant of  \algname{APM-C} \citep{rogozin2020towards,li2018sharp}. We do not perform experiments with \algname{ADOM}, because it is a dual based algorithm, and because  its empirical behavior was studied in \citep{kovalev2021adom}. 

For each condition number of the problem $\kappa \in \{10,10^2,10^3,10^4\}$, and condition number of the time-varying network  $\chi \in \{3,8,37,223,2704,4628\}$, we perform a comparison of these algorithms. Figures~\ref{fig:comm} and~\ref{fig:grad} show the convergence of the algorithms in the number of decentralized communications and the number of local computations, for all chosen values of $\kappa$ and $\chi$, respectively. 

Overall, the results are similar to what was obtained in Section~\ref{sec:experiments}. We observe that \algname{ADOM+} and \algname{Acc-GT} have similar behavior, which is expected since they are both optimal. Both of them perform better than \algname{APM-C} in terms of the number of decentralized communication rounds, which is expected since \algname{APM-C} has a sublinear communication complexity only. However, \algname{APM-C} performs better in terms of the number of local computations. Indeed, while all three algorithms are optimal in local computation complexity,\algname{APM-C} has better constants.

\clearpage
\begin{figure}[H]
	\includegraphics[width=\textwidth]{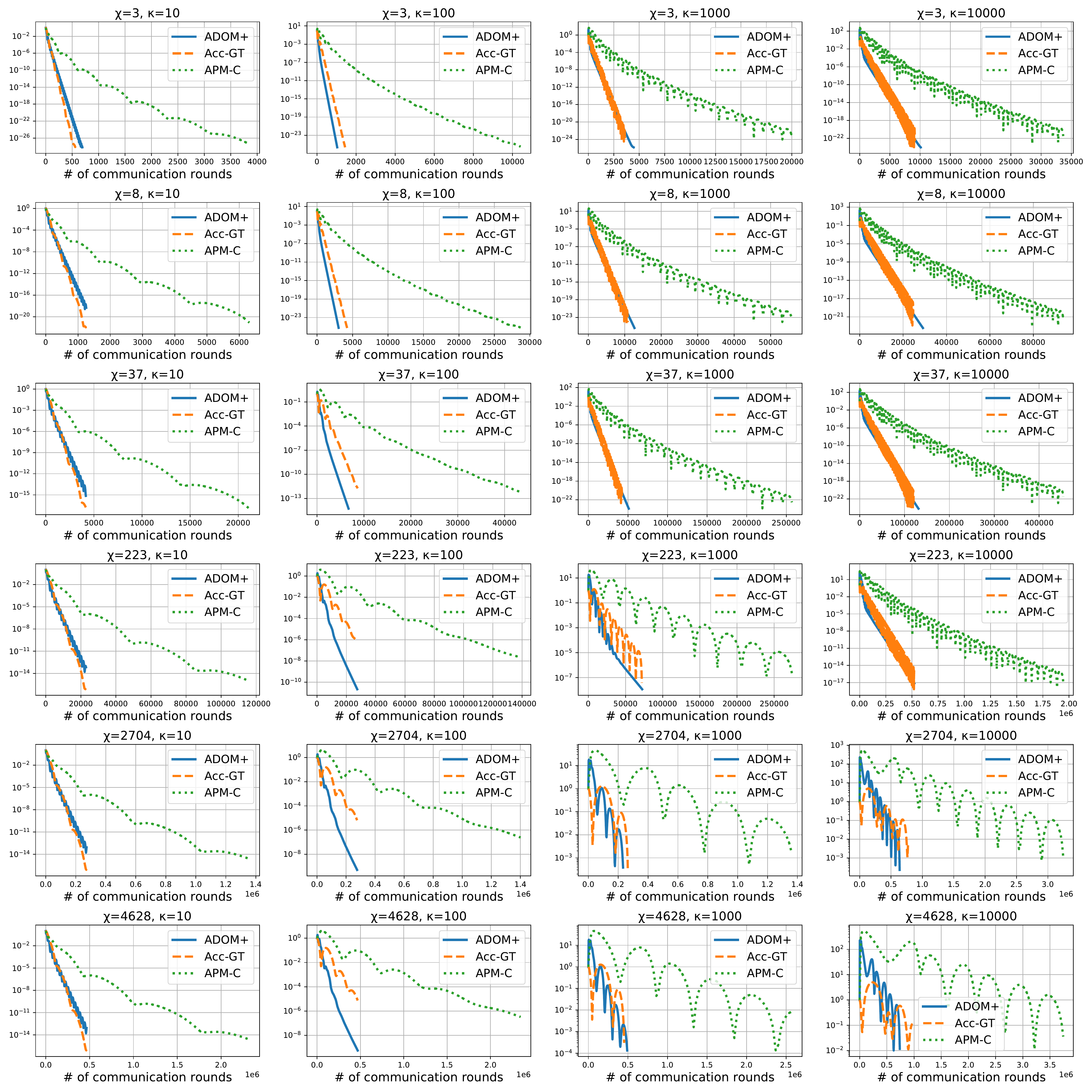}
	\caption{Comparison of our method \algname{ADOM+} and the baselines  \algname{Acc-GT} and \algname{APM-C} in {\bf decentralized communication complexity} on problems with $\kappa \in \{10,10^2,10^3,10^4\}$ and time-varying networks with $\chi \in \{3,8,37,223,2704,4628\}$.}
	\label{fig:comm}
\end{figure}

\begin{figure}[H]
	\includegraphics[width=\textwidth]{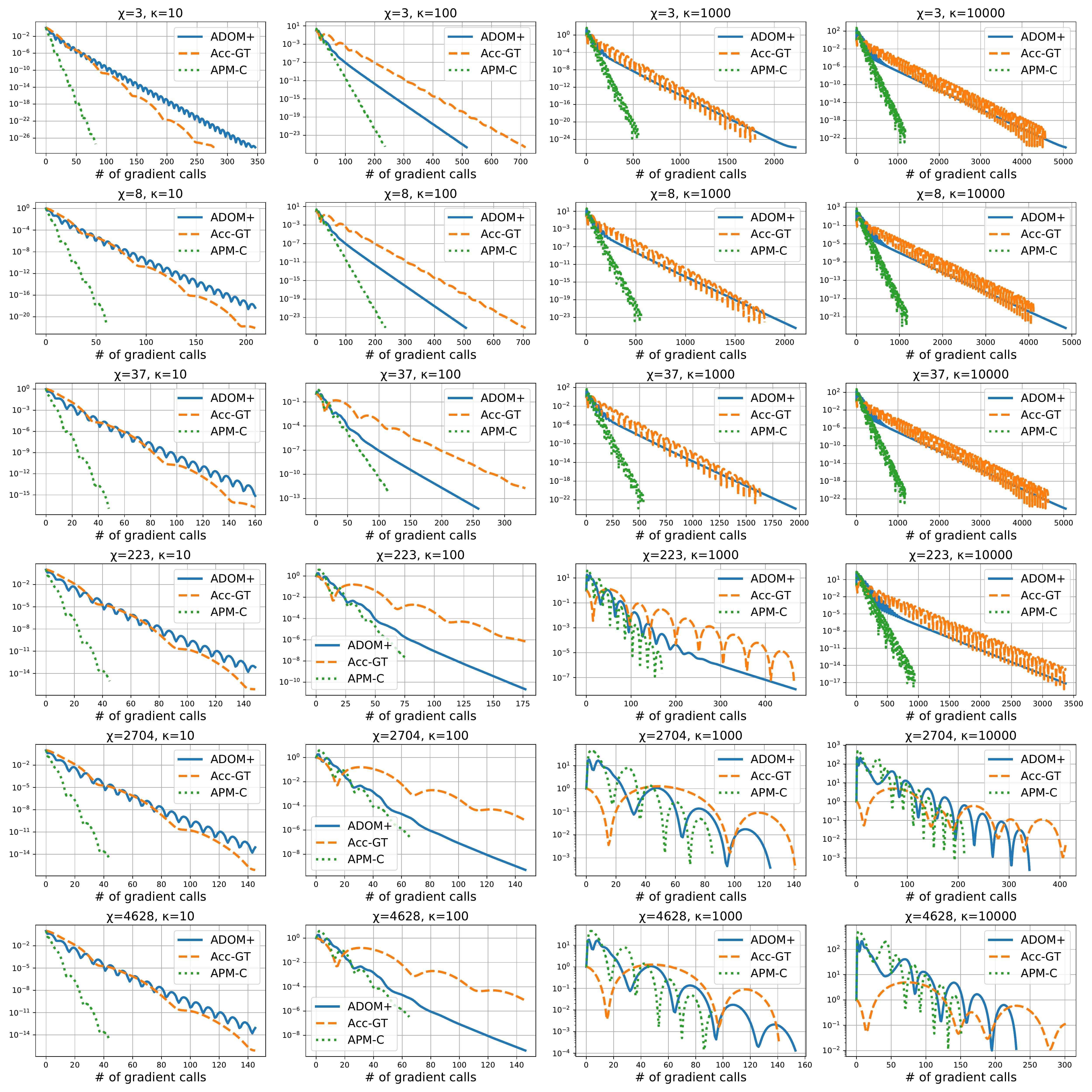}
	\caption{Comparison of our method \algname{ADOM+} and the baselines  \algname{Acc-GT} and \algname{APM-C} in {\bf local computation complexity} on problems with $\kappa \in \{10,10^2,10^3,10^4\}$ and time-varying networks with $\chi \in \{3,8,37,223,2704,4628\}$.}
	\label{fig:grad}
\end{figure}

\newpage

\section{Proof of Theorem~\ref{lower:thm:comm}}

\begin{proof}
	We choose number of nodes $n = 3 \floor{\chi / 3}$. Hence $n\geq 3$ and $n \mod 3 = 0$. Now, we divide the set of nodes $\cV = \{1,\ldots,n\}$ into three disjoint sets $\cV = \cV_1 \cup \cV_2 \cup \cV_3$: $\cV_1 = \{1,\ldots, n/3\}$, $\cV_2 = \{n/3 + 1,\ldots,2n/3\}$, $\cV_3 = \{2n/3 + 1, \ldots,n\}$. Note, that $|\cV_1| = |\cV_2| = |\cV_3| = n/3$.
	
	We define $\cG^q$ to be a star graph centered at the node $i_c(q) = |\cV_1| + 1 + (q \mod |\cV_2|)$. Note, that $i_c(q) \in \cV_2$ for all $q \in \{0,1,2,\ldots\}$. We define $\mW(q)$ as the Laplacian matrix of the graph $\cG(q)$. Hence, $\lmax(\mW(q)) /\lminp(\mW(q)) =  n \leq \chi$ and $\mW(q)$ satisfies condition \eqref{eq:chi}.
	
	We define functions $f_i$ in the following way:
	\begin{equation}\label{lower:f}
	f_i(x) = \begin{cases}
	\frac{\mu}{2}\sqn{x} + \frac{L-\mu}{4}\left[(x_{[1]} - 1)^2 + \sum_{l=1}^\infty(x_{[2l]} - x_{[2l+1]})^2\right],& i \in \cV_1\\
	\frac{\mu}{2}\sqn{x},& i \in \cV_2\\
	\frac{\mu}{2}\sqn{x} + \frac{L-\mu}{4} \sum_{l=1}^\infty(x_{[2l - 1]} - x_{[2l]})^2,& i \in \cV_3
	\end{cases}.
	\end{equation}
	The following lemma gives the solution of  problem \eqref{eq:main} with such a choice of $f_i$.
	\begin{lemma}\label{lower:lem:solution}
		Problem \eqref{eq:main}  with $f_i$ given by \eqref{lower:f} has a unique 	solution $x^* = (\rho^l)_{l=1}^\infty\in \ell_2$, where $\rho$ is given by 
		\begin{equation}\label{lower:rho}
		\rho = \frac{\sqrt{\frac{2L}{3\mu} + \frac{1}{3}} - 1}{\sqrt{\frac{2L}{3\mu} + \frac{1}{3}} + 1}.
		\end{equation}
	\end{lemma}
	
	Consider the following quantity:
	\begin{equation}
	s_i(k) = \begin{cases}
	0, &\cH_i(k) \subseteq \{0\},\\
	\min\{s \in \{1,2,\ldots\} : \cH_i(k) \subset \Span(\{e_1,e_2,\ldots,e_s\})\}, & \text{otherwise}
	\end{cases},
	\end{equation}
	where $e_s$ is the $s$-th unit basis vector.
	Using assumptions on the algorithm from Section~\ref{sec:alg_class} we can make some conclusions about update of $s_i(k)$. In particular, if at time step $k$ algorithm performs a local computation round, then using the fact, that $f_i(x)$ is a quadratic function with a certain block structure (and $f^*(x)$ has the same block structure), one can observe that
	\begin{equation}\label{lower:s_update_comp}
	s_i(k+1) \leq s_i(k) + \begin{cases}
	1 - (s_i(k)  \mod 2),  &  i \in \cV_1\\
	0,  &  i \in \cV_2\\
	(s_i(k)  \mod 2),  &  i \in \cV_3
	\end{cases}.
	\end{equation}
	Similarly, if at time step $k$ communication round number $q$ was performed, then using the structure of $\cG^q$ one can observe that	\begin{equation}\label{lower:s_update_comm}
	s_i(k+1) \leq \begin{cases}
	\max\{s_i(k), s_{i_c(q)}(k)\}, & i \neq i_c(q)\\
	\max\{s_j(k) : j \in \cV\}, & i =i_c(q)
	\end{cases}.
	\end{equation}
	The next key lemma shows that $s_i(k)$ is bounded compared to the number of communication rounds.
	\begin{lemma}\label{lower:lem:comm}
		Let $k\in \{0,1,\ldots, k_q\}$ be any time step, where $k_q$ is a time step at which the algorithm performed communication round number $q \in \{0,1,2,\ldots\}$. Then the following statement is true:
		\begin{equation}\label{lower:eq:comm}
		s_i(k) \leq 2 \floor{q / |\cV_2|} + \begin{cases}
		0, & i \in \cV_3 \cup \{i_c(q),i_c(q) + 1, \ldots,|\cV_1| + |\cV_2|\}\\
		1, & \text{otherwise}
		\end{cases}.
		\end{equation}
	\end{lemma}
	Lemma~\ref{lower:lem:comm} implies
	\begin{equation*}
	s_i(k) \leq \frac{2q}{|\cV_2|}+1 = \frac{6q}{n} + 1 
	=
	\frac{2q}{\floor{\chi /3}} + 1 
	\leq
	\frac{12q}{\chi} + 1.
	\end{equation*}
	Hence, using Lemma~\eqref{lower:lem:solution} we can lower bound $\sqn{x_i^o(k) - x^*}$:
	\begin{align*}
	\sqn{x_i^o(k) - x^*}
	&=
	\sum_{l=1}^\infty (x_i^o(k) - x^*)_{[l]}^2
	\geq
	\sum_{l=s_i(k) + 1}^\infty (x_i^o(k) - x^*)_{[l]}^2
	\\&=
	\sum_{l=s_i(k) + 1}^\infty \rho^{2l}
	=
	\rho^{2s_i(k) + 2} \sum_{l=0}^\infty \rho^{2l}
	=
	\frac{\rho^{2s_i(k) + 2}}{1-\rho^2}
	\\&\geq
	\frac{\rho^{24q/\chi + 4}}{1-\rho^2} \geq C\rho^{24q/\chi},
	\end{align*}
	where $C = \frac{\rho^4}{1-\rho^2}$. Note, that $\rho \geq \max\left\{0,1 - \frac{\sqrt{6\mu}}{\sqrt{L}} \right\}$ and hence
	\begin{equation*}
	\sqn{x_i^o(k) - x^*} \geq C \left(\max\left\{0,1 - \frac{\sqrt{6\mu}}{\sqrt{L}} \right\}\right)^{24q/\chi}.
	\end{equation*}
	Finally, using Bernoulli inequality we get
	\begin{equation*}
	\sqn{x_i^o(k) - x^*} \geq C \left(\max\left\{0,1 - \frac{24\sqrt{6\mu}}{\sqrt{L}} \right\}\right)^{q/\chi}.
	\end{equation*}

\end{proof}   

\begin{proof}[Proof of Lemma~\ref{lower:lem:solution}]
	\begin{align*}
	\sum_{i \in \cV} f_i(x)
	&=
	\frac{n\mu}{2}\sqn{x} + \frac{n(L-\mu)}{12}\left[(x_{[1]} - 1)^2 + \sum_{l=1}^\infty(x_{[l]} - x_{[l+1]})^2\right]
	\\&=
	\frac{n\mu}{2}\sqn{x}
	+ 
	\frac{n(L-\mu)}{12}\left[x_{[1]}^2 - 2x_{[1]} + 1 + \sum_{l=1}^\infty(x_{[l]}^2 - 2x_{[l]} x_{[l+1]} + x_{[l+1]}^2)\right]
	\\&=
	\frac{n\mu}{2}\sqn{x}
	+ 
	\frac{n(L-\mu)}{12}\left[\sum_{l=1}^\infty(2x_{[l]}^2 - 2x_{[l]} x_{[l+1]}) - 2x_{[1]} + 1 \right]
	\\&=
	\frac{n(L-\mu)}{12}\left[\sum_{l=1}^\infty\left(\left(2 + \frac{6\mu}{L-\mu}\right)x_{[l]}^2 - 2x_{[l]} x_{[l+1]}\right) - 2x_{[1]} + 1 \right]
	\\&=
	\frac{n(L-\mu)}{12}\left[\sum_{l=1}^\infty\left(\frac{2(L+2\mu)}{L-\mu}x_{[l]}^2 - 2x_{[l]} x_{[l+1]}\right) - 2x_{[1]} + 1 \right].
	\end{align*}
	Next, we use the fact that $\frac{2(L+2\mu)}{L-\mu} = \rho + \frac{1}{\rho}$ with $\rho$ given by \eqref{lower:rho} and get
	\begin{align*}
	\sum_{i \in \cV} f_i(x)
	&=
	\frac{n(L-\mu)}{12}\left[\sum_{l=1}^\infty\left(\rho x_{[l]}^2  + \frac{1}{\rho}x_{[l]}^2 - 2x_{[l]} x_{[l+1]}\right) - 2x_{[1]} + 1 \right]
	\\&=
	\frac{n(L-\mu)}{12\rho}\left[\sum_{l=1}^\infty\left(\rho^2 x_{[l]}^2  + x_{[l]}^2 - 2\rho x_{[l]} x_{[l+1]}\right) - 2\rho x_{[1]} + \rho \right]
	\\&=
	\frac{n(L-\mu)}{12\rho}\left[\sum_{l=1}^\infty\left(\rho^2 x_{[l]}^2  - 2\rho x_{[l]} x_{[l+1]} + x_{[l+1]}^2 \right)+ x_{[1]}^2 - 2\rho x_{[1]} + \rho \right]
	\\&=
	\frac{n(L-\mu)}{12\rho}\left[\sum_{l=1}^\infty(\rho x_{[l]} - x_{[l+1]})^2+ (x_{[1]}- \rho )^2 + \rho - \rho^2 \right].
	\end{align*}
	Now it's clear that $\sum_{i \in \cV} f_i(x) \geq \frac{n(L-\mu)(1-\rho)}{12}$ and $\sum_{i \in \cV} f_i(x) = \frac{n(L-\mu)(1-\rho)}{12}$ if and only if $x= x^*$. Hence, $x^*$ is indeed a unique solution to the problem \eqref{eq:main}.
\end{proof}                

\begin{proof}[Proof of Lemma~\ref{lower:lem:comm}]
	We prove this by induction in $q$.
	
	\paragraph{Induction basis.} When $q = 0$ and $k \in \{0,1,\ldots, k_0 \}$ (meaning no communication rounds were done), from \eqref{lower:s_update_comp} we can conclude, that
	\begin{equation*}
	s_i(k) \leq \begin{cases}
	1, & i  \in \cV_1\\
	0 , & i \in \cV_2 \cup \cV_3
	\end{cases},
	\end{equation*}
	which is $\eqref{lower:eq:comm}$ in case $q = 0$.
	
	\paragraph{Induction step.} Now, we assume that \eqref{lower:eq:comm} holds for $q$ and $k \in \{0,1,\ldots,k_q\}$ and prove it for $q + 1$ and $k \in \{0,1,\ldots,k_{q+1}\}$. Indeed, consider time step $k_q$ at which communication round $q$ is performed. Consider two possible cases:
	\begin{enumerate}
		\item $i_c(q) \neq |\cV_1| + |\cV_2|$.\\
		In this case $\floor{(q+1) / \cV_2} = \floor{q / \cV_2}$ and $i_c(q+1)= i_c(q) + 1$. Since \eqref{lower:eq:comm} holds for $q$ and $k_q$, using \eqref{lower:s_update_comm} we get $s_{i_c(q)}(k_q) \leq 2\floor{q / \cV_2}$ and  $s_{i_c(q)}(k_q+1) \leq 2\floor{q / \cV_2} + 1$. Hence, using \eqref{lower:s_update_comm} we get for all $i \in \cV$
		\begin{align*}
		s_i(k) &\leq 2 \floor{q / |\cV_2|} + \begin{cases}
		0, & i \in \cV_3 \cup \{i_c(q) + 1, \ldots,|\cV_1| + |\cV_2|\}\\
		1, & \text{otherwise}
		\end{cases}
		\\&=
		2 \floor{(q+1) / |\cV_2|} + \begin{cases}
		0, & i \in \cV_3 \cup \{i_c(q) + 1, \ldots,|\cV_1| + |\cV_2|\}\\
		1, & \text{otherwise}
		\end{cases},
		\end{align*}
		which is \eqref{lower:eq:comm} for $q+1$ and $k = k_q + 1$.
		
		\item $i_c(q) = |\cV_1| + |\cV_2|$.\\
		In this case $\floor{(q+1) / \cV_2} = \floor{q / \cV_2} + 1$ and $i_c(q+1)= |\cV_1| + 1$. Since \eqref{lower:eq:comm} holds for $q$ and $k_q$, we have an upper bound $s_i(k_q) \leq 2\floor{q / \cV_2} + 1$ and \eqref{lower:s_update_comm} implies 
		\begin{align*}
		s_i(k_q + 1) &\leq 2\floor{q / \cV_2} + 1 \leq 2(\floor{q / \cV_2} + 1) = 2\floor{(q+1) / \cV_2}
		\\&\leq
		2 \floor{(q+1) / |\cV_2|} + \begin{cases}
		0, & i \in \cV_3 \cup \{i_c(q+1), \ldots,|\cV_1| + |\cV_2|\}\\
		1, & \text{otherwise}
		\end{cases}.
		\end{align*}
		which is \eqref{lower:eq:comm} for $q+1$ and $k = k_q + 1$.
	\end{enumerate}
	In both cases we have obtained \eqref{lower:eq:comm} for $q+1$ and $k = k_q + 1$. It remains to see that at any time step $k \in \{k_q + 2, k_{q+1} - 1\}$ only local computation is performed and \eqref{lower:s_update_comp} implies that \eqref{lower:eq:comm} also holds for $q +1$ and any $k \in \{0,1,\ldots, k_{q+1}\}$.
	
\end{proof}

\newpage

\section{Proof of Theorem~\ref{thm:adom+_conv}}

By $\bg_F(x,y)$ we denote Bregman distance $\bg_F(x,y)\eqdef F(x) - F(y) - \<\g F(y),x-y>$.

\begin{lemma}
	Let $\tau_2$ be defined as follows:
	\begin{equation}\label{scary:tau2}
	\tau_2 = \sqrt{\mu/L}.
	\end{equation}
	Let $\tau_1$ be defined as follows:
	\begin{equation}\label{scary:tau1}
	\tau_1 = (1/\tau_2 + 1/2)^{-1}.
	\end{equation}
	Let $\eta$ be defined as follows:
	\begin{equation}\label{scary:eta}
	\eta = (L\tau_2)^{-1}.
	\end{equation}
	Let $\alpha$ be defined as follows:
	\begin{equation}\label{scary:alpha}
	\alpha = \mu/2.
	\end{equation}
	Let $\nu$ be defined as follows:
	\begin{equation}\label{scary:nu}
	\nu = \mu/2.
	\end{equation}
	Let $\Psi_x^k$ be defined as follows:
	\begin{equation}\label{scary:Psi_x}
	\Psi_x^k = \left(\frac{1}{\eta} + \alpha\right)\sqn{x^{k} - x^*} + \frac{2}{\tau_2}\left(\bg_f(x_f^{k},x^*)-\frac{\nu}{2}\sqn{x_f^{k} - x^*} \right)
	\end{equation}
	Then the following inequality holds:
	\begin{equation}\label{scary:eq:x}
	\Psi_x^{k+1} \leq \left(1 - \frac{\sqrt{\mu}}{\sqrt{\mu}+2\sqrt{L}}\right)\Psi_x^k
	+
	2\< y^{k+1} - y^*,x^{k+1} - x^*>
	-
	\left(\bg_F(x_g^k,x^*) - \frac{\nu}{2}\sqn{x_g^k - x^*}\right).
	\end{equation}
\end{lemma}
\begin{proof}
	\begin{align*}
	\frac{1}{\eta}\sqn{x^{k+1}  - x^*}
	&=
	\frac{1}{\eta}\sqn{x^k - x^*}+\frac{2}{\eta}\<x^{k+1} - x^k,x^{k+1}- x^*> - \frac{1}{\eta}\sqn{x^{k+1} - x^k}.
	\end{align*}
	Using Line~\ref{scary:line:x:2} of Algorithm~\ref{scary:alg} we get
	\begin{align*}
	\frac{1}{\eta}\sqn{x^{k+1}  - x^*}
	&=
	\frac{1}{\eta}\sqn{x^k - x^*}
	+
	2\alpha\<x_g^k - x^{k+1},x^{k+1}- x^*>
	\\&-
	2\<\g F(x_g^k) - \nu x_g^k - y^{k+1},x^{k+1} - x^*>
	-
	\frac{1}{\eta}\sqn{x^{k+1} - x^k}
	\\&=
	\frac{1}{\eta}\sqn{x^k - x^*}
	+
	2\alpha\<x_g^k - x^*- x^{k+1} + x^*,x^{k+1}- x^*>
	\\&-
	2\<\g F(x_g^k) - \nu x_g^k - y^{k+1},x^{k+1} - x^*>
	-
	\frac{1}{\eta}\sqn{x^{k+1} - x^k}
	\\&\leq
	\frac{1}{\eta}\sqn{x^k - x^*}
	-
	\alpha\sqn{x^{k+1} - x^*} + \alpha\sqn{x_g^k - x^*}
	-
	2\<\g F(x_g^k) - \nu x_g^k - y^{k+1},x^{k+1} - x^*>
	\\&-
	\frac{1}{\eta}\sqn{x^{k+1} - x^k}.
	\end{align*}
	Using optimality condition \eqref{opt:x} we get
	\begin{align*}
	\frac{1}{\eta}\sqn{x^{k+1}  - x^*}
	&\leq
	\frac{1}{\eta}\sqn{x^k - x^*}
	-
	\alpha\sqn{x^{k+1} - x^*} + \alpha\sqn{x_g^k - x^*}
	-
	\frac{1}{\eta}\sqn{x^{k+1} - x^k}
	\\&
	-2\<\g F(x_g^k) - \g F(x^*),x^{k+1} - x^*>
	+
	2\nu\< x_g^k - x^*,x^{k+1} - x^*>
	+
	2\< y^{k+1} - y^*,x^{k+1} - x^*>.
	\end{align*}
	Using Line~\ref{scary:line:x:3} of Algorithm~\ref{scary:alg} we get
	\begin{align*}
	\frac{1}{\eta}\sqn{x^{k+1}  - x^*}
	&\leq
	\frac{1}{\eta}\sqn{x^k - x^*}
	-
	\alpha\sqn{x^{k+1} - x^*} + \alpha\sqn{x_g^k - x^*}
	-
	\frac{1}{\eta\tau_2^2}\sqn{x_f^{k+1} - x_g^k}
	\\&
	-2\<\g F(x_g^k) - \g F(x^*),x^k - x^*>
	+
	2\nu\< x_g^k - x^*,x^k - x^*>
	+
	2\< y^{k+1} - y^*,x^{k+1} - x^*>
	\\&-
	\frac{2}{\tau_2}\<\g F(x_g^k) - \g F(x^*),x_f^{k+1} - x_g^k>
	+
	\frac{2\nu}{\tau_2}\< x_g^k - x^*,x_f^{k+1} - x_g^k>
	\\&=
	\frac{1}{\eta}\sqn{x^k - x^*}
	-
	\alpha\sqn{x^{k+1} - x^*} + \alpha\sqn{x_g^k - x^*}
	-
	\frac{1}{\eta\tau_2^2}\sqn{x_f^{k+1} - x_g^k}
	\\&
	-2\<\g F(x_g^k) - \g F(x^*),x^k - x^*>
	+
	2\nu\< x_g^k - x^*,x^k - x^*>
	+
	2\< y^{k+1} - y^*,x^{k+1} - x^*>
	\\&-
	\frac{2}{\tau_2}\<\g F(x_g^k) - \g F(x^*),x_f^{k+1} - x_g^k>
	+
	\frac{\nu}{\tau_2}\left(\sqn{x_f^{k+1} - x^*} - \sqn{x_g^k - x^*}-\sqn{x_f^{k+1} - x_g^k}\right).
	\end{align*}
	Using $L$-smoothness of $\bg_F(x,x^*)$ in $x$, which follows from $L$-smoothness of $F(x)$, we get
	\begin{align*}
	\frac{1}{\eta}\sqn{x^{k+1}  - x^*}
	&\leq
	\frac{1}{\eta}\sqn{x^k - x^*}
	-
	\alpha\sqn{x^{k+1} - x^*} + \alpha\sqn{x_g^k - x^*}
	-
	\frac{1}{\eta\tau_2^2}\sqn{x_f^{k+1} - x_g^k}
	\\&
	-2\<\g F(x_g^k) - \g F(x^*),x^k - x^*>
	+
	2\nu\< x_g^k - x^*,x^k - x^*>
	+
	2\< y^{k+1} - y^*,x^{k+1} - x^*>
	\\&-
	\frac{2}{\tau_2}\<\g F(x_g^k) - \g F(x^*),x_f^{k+1} - x_g^k>
	+
	\frac{\nu}{\tau_2}\left(\sqn{x_f^{k+1} - x^*} - \sqn{x_g^k - x^*}-\sqn{x_f^{k+1} - x_g^k}\right)
	\\&\leq
	\frac{1}{\eta}\sqn{x^k - x^*}
	-
	\alpha\sqn{x^{k+1} - x^*} + \alpha\sqn{x_g^k - x^*}
	-
	\frac{1}{\eta\tau_2^2}\sqn{x_f^{k+1} - x_g^k}
	\\&
	-2\<\g F(x_g^k) - \g F(x^*),x^k - x^*>
	+
	2\nu\< x_g^k - x^*,x^k - x^*>
	+
	2\< y^{k+1} - y^*,x^{k+1} - x^*>
	\\&-
	\frac{2}{\tau_2}\left(\bg_f(x_f^{k+1},x^*) - \bg_f(x_g^k,x^*) - \frac{L}{2}\sqn{x_f^{k+1} - x_g^k}\right)
	\\&+
	\frac{\nu}{\tau_2}\left(\sqn{x_f^{k+1} - x^*} - \sqn{x_g^k - x^*}-\sqn{x_f^{k+1} - x_g^k}\right)
	\\&=
	\frac{1}{\eta}\sqn{x^k - x^*}
	-
	\alpha\sqn{x^{k+1} - x^*} + \alpha\sqn{x_g^k - x^*}
	+
	\left(\frac{L - \nu}{\tau_2}-\frac{1}{\eta\tau_2^2}\right)
	\sqn{x_f^{k+1} - x_g^k}
	\\&
	-2\<\g F(x_g^k) - \g F(x^*),x^k - x^*>
	+
	2\nu\< x_g^k - x^*,x^k - x^*>
	+
	2\< y^{k+1} - y^*,x^{k+1} - x^*>
	\\&-
	\frac{2}{\tau_2}\left(\bg_f(x_f^{k+1},x^*) - \bg_f(x_g^k,x^*)\right)
	+
	\frac{\nu}{\tau_2}\left(\sqn{x_f^{k+1} - x^*} - \sqn{x_g^k - x^*}\right)
	\end{align*}
	Using Line~\ref{scary:line:x:1} of Algorithm~\ref{scary:alg} we get
	\begin{align*}
	\frac{1}{\eta}\sqn{x^{k+1}  - x^*}
	&\leq
	\frac{1}{\eta}\sqn{x^k - x^*}
	-
	\alpha\sqn{x^{k+1} - x^*} + \alpha\sqn{x_g^k - x^*}
	+
	\left(\frac{L - \nu}{\tau_2}-\frac{1}{\eta\tau_2^2}\right)
	\sqn{x_f^{k+1} - x_g^k}
	\\&
	-2\<\g F(x_g^k) - \g F(x^*),x_g^k - x^*>
	+
	2\nu\sqn{x_g^k - x^*}
	+
	\frac{2(1-\tau_1)}{\tau_1}\<\g F(x_g^k) - \g F(x^*),x_f^k - x_g^k>
	\\&+
	\frac{2\nu(1-\tau_1)}{\tau_1}\<x_g^k - x_f^k,x_g^k - x^*>
	+
	2\< y^{k+1} - y^*,x^{k+1} - x^*>
	\\&-
	\frac{2}{\tau_2}\left(\bg_f(x_f^{k+1},x^*) - \bg_f(x_g^k,x^*)\right)
	+
	\frac{\nu}{\tau_2}\left(\sqn{x_f^{k+1} - x^*} - \sqn{x_g^k - x^*}\right)
	\\&=
	\frac{1}{\eta}\sqn{x^k - x^*}
	-
	\alpha\sqn{x^{k+1} - x^*} + \alpha\sqn{x_g^k - x^*}
	+
	\left(\frac{L - \nu}{\tau_2}-\frac{1}{\eta\tau_2^2}\right)
	\sqn{x_f^{k+1} - x_g^k}
	\\&
	-2\<\g F(x_g^k) - \g F(x^*),x_g^k - x^*>
	+
	2\nu\sqn{x_g^k - x^*}
	+
	\frac{2(1-\tau_1)}{\tau_1}\<\g F(x_g^k) - \g F(x^*),x_f^k - x_g^k>
	\\&+
	\frac{\nu(1-\tau_1)}{\tau_1}\left(\sqn{x_g^k- x_f^k} + \sqn{x_g^k - x^*} - \sqn{x_f^k - x^*}\right)
	+
	2\< y^{k+1} - y^*,x^{k+1} - x^*>
	\\&-
	\frac{2}{\tau_2}\left(\bg_f(x_f^{k+1},x^*) - \bg_f(x_g^k,x^*)\right)
	+
	\frac{\nu}{\tau_2}\left(\sqn{x_f^{k+1} - x^*} - \sqn{x_g^k - x^*}\right).
	\end{align*}
	Using $\mu$-strong convexity of $\bg_F(x,x^*)$ in $x$, which follows from $\mu$-strong convexity of $F(x)$, we get
	\begin{align*}
	\frac{1}{\eta}\sqn{x^{k+1}  - x^*}
	&\leq
	\frac{1}{\eta}\sqn{x^k - x^*}
	-
	\alpha\sqn{x^{k+1} - x^*} + \alpha\sqn{x_g^k - x^*}
	+
	\left(\frac{L - \nu}{\tau_2}-\frac{1}{\eta\tau_2^2}\right)
	\sqn{x_f^{k+1} - x_g^k}
	\\&
	-2\bg_F(x_g^k,x^*) - \mu\sqn{x_g^k - x^*}
	+
	2\nu\sqn{x_g^k - x^*}
	\\&+
	\frac{2(1-\tau_1)}{\tau_1}\left(\bg_F(x_f^k,x^*) - \bg_F(x_g^k,x^*) - \frac{\mu}{2}\sqn{x_f^k - x_g^k}\right)
	\\&+
	\frac{\nu(1-\tau_1)}{\tau_1}\left(\sqn{x_g^k- x_f^k} + \sqn{x_g^k - x^*} - \sqn{x_f^k - x^*}\right)
	+
	2\< y^{k+1} - y^*,x^{k+1} - x^*>
	\\&-
	\frac{2}{\tau_2}\left(\bg_f(x_f^{k+1},x^*) - \bg_f(x_g^k,x^*)\right)
	+
	\frac{\nu}{\tau_2}\left(\sqn{x_f^{k+1} - x^*} - \sqn{x_g^k - x^*}\right)
	\\&=
	\frac{1}{\eta}\sqn{x^k - x^*}
	-
	\alpha\sqn{x^{k+1} - x^*}
	+
	\frac{2(1-\tau_1)}{\tau_1}\left(\bg_F(x_f^k,x^*) - \frac{\nu}{2}\sqn{x_f^k - x^*}\right)
	\\&-
	\frac{2}{\tau_2}\left(\bg_f(x_f^{k+1},x^*)-\frac{\nu}{2}\sqn{x_f^{k+1} - x^*} \right)
	+
	2\< y^{k+1} - y^*,x^{k+1} - x^*>
	\\&+
	2\left(\frac{1}{\tau_2}-\frac{1}{\tau_1}\right)\bg_F(x_g^k,x^*)
	+
	\left(\alpha - \mu + \nu+\frac{\nu}{\tau_1}-\frac{\nu}{\tau_2}\right)\sqn{x_g^k - x^*}
	\\&+
	\left(\frac{L - \nu}{\tau_2}-\frac{1}{\eta\tau_2^2}\right)
	\sqn{x_f^{k+1} - x_g^k}
	+
	\frac{(1-\tau_1)(\nu-\mu)}{\tau_1}\sqn{x_f^k - x_g^k}.
	\end{align*}
	Using $\eta$ defined by \eqref{scary:eta}, $\tau_1$ defined by \eqref{scary:tau1} and the fact that $\nu < \mu$ we get
	\begin{align*}
	\frac{1}{\eta}\sqn{x^{k+1}  - x^*}
	&\leq
	\frac{1}{\eta}\sqn{x^k - x^*}
	-
	\alpha\sqn{x^{k+1} - x^*}
	+
	\frac{2(1-\tau_2/2)}{\tau_2}\left(\bg_F(x_f^k,x^*) - \frac{\nu}{2}\sqn{x_f^k - x^*}\right)
	\\&-
	\frac{2}{\tau_2}\left(\bg_f(x_f^{k+1},x^*)-\frac{\nu}{2}\sqn{x_f^{k+1} - x^*} \right)
	+
	2\< y^{k+1} - y^*,x^{k+1} - x^*>
	\\&-\bg_F(x_g^k,x^*)
	+
	\left(\alpha - \mu + \frac{3\nu}{2}\right)\sqn{x_g^k - x^*}.
	\end{align*}
	Using $\alpha$ defined by \eqref{scary:alpha} and $\nu$ defined by \eqref{scary:nu} we get
	\begin{align*}
	\frac{1}{\eta}\sqn{x^{k+1}  - x^*}
	&\leq
	\frac{1}{\eta}\sqn{x^k - x^*}
	-
	\alpha\sqn{x^{k+1} - x^*}
	+
	\frac{2(1-\tau_2/2)}{\tau_2}\left(\bg_F(x_f^k,x^*) - \frac{\nu}{2}\sqn{x_f^k - x^*}\right)
	\\&-
	\frac{2}{\tau_2}\left(\bg_f(x_f^{k+1},x^*)-\frac{\nu}{2}\sqn{x_f^{k+1} - x^*} \right)
	+
	2\< y^{k+1} - y^*,x^{k+1} - x^*>
	\\&-
	\left(\bg_F(x_g^k,x^*) - \frac{\nu}{2}\sqn{x_g^k - x^*}\right).
	\end{align*}
	After rearranging and using $\Psi_x^k $ definition \eqref{scary:Psi_x} we get
	\begin{align*}
	\Psi_x^{k+1}
	&\leq
	\max\left\{1 - \tau_2/2, 1/(1+\eta\alpha)\right\}\Psi_x^k
	+
	2\< y^{k+1} - y^*,x^{k+1} - x^*>
	-
	\left(\bg_F(x_g^k,x^*) - \frac{\nu}{2}\sqn{x_g^k - x^*}\right)
	\\&\leq 
	\left(1 - \frac{\sqrt{\mu}}{\sqrt{\mu}+2\sqrt{L}}\right)\Psi_x^k
	+
	2\< y^{k+1} - y^*,x^{k+1} - x^*>
	-
	\left(\bg_F(x_g^k,x^*) - \frac{\nu}{2}\sqn{x_g^k - x^*}\right).
	\end{align*}
\end{proof}

\begin{lemma}
	The following inequality holds:
	\begin{equation}
	\begin{split}\label{scary:eq:1}
	-\sqn{y^{k+1}- y^*}
	&\leq 
	\frac{(1-\sigma_1)}{\sigma_1}\sqn{y_f^k - y^*}
	- \frac{1}{\sigma_2}\sqn{y_f^{k+1} - y^*}
	\\&-
	\left(\frac{1}{\sigma_1} - \frac{1}{\sigma_2}\right)\sqn{y_g^k - y^*}
	+
	\left(\sigma_2- \sigma_1\right)\sqn{y^{k+1} - y^k}
	\end{split}
	\end{equation}
\end{lemma}
\begin{proof}
	Lines~\ref{scary:line:y:1} and~\ref{scary:line:y:3} of Algorithm~\ref{scary:alg} imply
	\begin{align*}
	y_f^{k+1} &= y_g^k + \sigma_2(y^{k+1} - y_k)\\&=
	y_g^k + \sigma_2 y^{k+1} - \frac{\sigma_2}{\sigma_1}\left(y_g^k - (1-\sigma_1)y_f^k\right)
	\\&=
	\left(1 - \frac{\sigma_2}{\sigma_1}\right)y_g^k + \sigma_2 y^{k+1} + \left(\frac{\sigma_2}{\sigma_1}- \sigma_2\right) y_f^k.
	\end{align*}
	After subtracting $y^*$ and rearranging we get
	\begin{align*}
	(y_f^{k+1}- y^*)+ \left(\frac{\sigma_2}{\sigma_1} - 1\right) (y_g^k - y^*)
	=
	\sigma_2( y^{k+1} - y^*)+ \left(\frac{\sigma_2}{\sigma_1} - \sigma_2\right)(y_f^k - y^*).
	\end{align*}
	Multiplying both sides by $\frac{\sigma_1}{\sigma_2}$ gives
	\begin{align*}
	\frac{\sigma_1}{\sigma_2}(y_f^{k+1}- y^*)+ \left(1-\frac{\sigma_1}{\sigma_2}\right) (y_g^k - y^*)
	=
	\sigma_1( y^{k+1} - y^*)+ \left(1 - \sigma_1\right)(y_f^k - y^*).
	\end{align*}
	Squaring both sides gives
	\begin{align*}
	\frac{\sigma_1}{\sigma_2}\sqn{y_f^{k+1} - y^*} + \left(1- \frac{\sigma_1}{\sigma_2}\right)\sqn{y_g^k - y^*} - \frac{\sigma_1}{\sigma_2}\left(1-\frac{\sigma_1}{\sigma_2}\right)\sqn{y_f^{k+1} - y_g^k}
	\leq
	\sigma_1\sqn{y^{k+1} - y^*} + (1-\sigma_1)\sqn{y_f^k - y^*}.
	\end{align*}
	Rearranging gives
	\begin{align*}
	-\sqn{y^{k+1}- y^*} \leq -\left(\frac{1}{\sigma_1} - \frac{1}{\sigma_2}\right)\sqn{y_g^k - y^*} + \frac{(1-\sigma_1)}{\sigma_1}\sqn{y_f^k - y^*} - \frac{1}{\sigma_2}\sqn{y_f^{k+1} - y^*}+
	\frac{1}{\sigma_2}\left(1 - \frac{\sigma_1}{\sigma_2}\right)\sqn{y_f^{k+1} - y_g^k}.
	\end{align*}
	Using Line~\ref{scary:line:y:3} of Algorithm~\ref{scary:alg} we get
	\begin{align*}
	-\sqn{y^{k+1}- y^*} \leq -\left(\frac{1}{\sigma_1} - \frac{1}{\sigma_2}\right)\sqn{y_g^k - y^*} + \frac{(1-\sigma_1)}{\sigma_1}\sqn{y_f^k - y^*} - \frac{1}{\sigma_2}\sqn{y_f^{k+1} - y^*}+
	\left(\sigma_2 - \sigma_1\right)\sqn{y^{k+1} - y^k}.
	\end{align*}
\end{proof}
\newpage

\begin{lemma}
	Let $\beta$ be defined as follows:
	\begin{equation}\label{scary:beta}
	\beta = 1/(2L).
	\end{equation}
	Let $\sigma_1$ be defined as follows:
	\begin{equation}\label{scary:sigma1}
	\sigma_1 = (1/\sigma_2 + 1/2)^{-1}.
	\end{equation}
	Then the following inequality holds:
	\begin{equation}
	\begin{split}\label{scary:eq:y}
	\MoveEqLeft[4]
	\left(\frac{1}{\theta} + \frac{\beta}{2}\right)\sqn{y^{k+1} - y^*}
	+
	\frac{\beta}{2\sigma_2}\sqn{y_f^{k+1} - y^*}\\
	&\leq
	\frac{1}{\theta}\sqn{y^k - y^*}
	+
	\frac{\beta(1-\sigma_2/2)}{2\sigma_2}\sqn{y_f^k - y^*}
	+
	\bg_F(x_g^k, x^*) - \frac{\nu}{2}\sqn{x_g^k - x^*}
	-
	2\<x^{k+1} - x^*, y^{k+1} - y^*>
	\\&-
	2\nu^{-1}\<y_g^k + z_g^k - (y^* + z^*), y^{k+1} - y^*>
	-
	\frac{\beta}{4}\sqn{y_g^k - y^*}
	+
	\left(\frac{\beta\sigma_2^2}{4} - \frac{1}{\theta}\right)\sqn{y^{k+1} - y^k}.
	\end{split}
	\end{equation}
\end{lemma}
\begin{proof}
	\begin{align*}
	\frac{1}{\theta}\sqn{y^{k+1} - y^*}
	&=
	\frac{1}{\theta}\sqn{y^k - y^*} + \frac{2}{\theta}\<x^{k+1} - x^k , x^{k+1} - x^*> - \frac{1}{\theta}\sqn{y^{k+1} - y^k}. 
	\end{align*}
	Using Line~\ref{scary:line:y:2} of Algorithm~\ref{scary:alg} we get
	\begin{align*}
	\frac{1}{\theta}\sqn{y^{k+1} - y^*}
	&=
	\frac{1}{\theta}\sqn{y^k - y^*}
	+
	2\beta\<\g F(x_g^k) - \nu x_g^k - y^{k+1}, y^{k+1} - y^*>
	\\&-
	2\<\nu^{-1}(y_g^k + z_g^k) + x^{k+1}, y^{k+1} - y^*>
	-
	\frac{1}{\theta}\sqn{y^{k+1} - y^k}.
	\end{align*}
	Using optimality condition \eqref{opt:x} we get
	\begin{align*}
	\frac{1}{\theta}\sqn{y^{k+1} - y^*}
	&=
	\frac{1}{\theta}\sqn{y^k - y^*}
	+
	2\beta\<\g F(x_g^k) - \nu x_g^k - (\g F(x^*) - \nu x^*) + y^*- y^{k+1}, y^{k+1} - y^*>
	\\&-
	2\<\nu^{-1}(y_g^k + z_g^k) + x^{k+1}, y^{k+1} - y^*>
	-
	\frac{1}{\theta}\sqn{y^{k+1} - y^k}
	\\&=
	\frac{1}{\theta}\sqn{y^k - y^*}
	+
	2\beta\<\g F(x_g^k) - \nu x_g^k - (\g F(x^*) - \nu x^*), y^{k+1} - y^*>
	-
	2\beta\sqn{y^{k+1} - y^*}
	\\&-
	2\<\nu^{-1}(y_g^k + z_g^k) + x^{k+1}, y^{k+1} - y^*>
	-
	\frac{1}{\theta}\sqn{y^{k+1} - y^k}
	\\&\leq
	\frac{1}{\theta}\sqn{y^k - y^*}
	+
	\beta\sqn{\g F(x_g^k) - \nu x_g^k - (\g F(x^*) - \nu x^*)}
	-
	\beta\sqn{y^{k+1} - y^*}
	\\&-
	2\<\nu^{-1}(y_g^k + z_g^k) + x^{k+1}, y^{k+1} - y^*>
	-
	\frac{1}{\theta}\sqn{y^{k+1} - y^k}.
	\end{align*}
	Function $F(x) - \frac{\nu}{2}\sqn{x}$ is convex and $L$-smooth, which implies
	\begin{align*}
	\frac{1}{\theta}\sqn{y^{k+1} - y^*}
	&\leq
	\frac{1}{\theta}\sqn{y^k - y^*}
	+
	2\beta L\left(\bg_F(x_g^k, x^*) - \frac{\nu}{2}\sqn{x_g^k - x^*}\right)
	-
	\beta\sqn{y^{k+1} - y^*}
	\\&-
	2\<\nu^{-1}(y_g^k + z_g^k) + x^{k+1}, y^{k+1} - y^*>
	-
	\frac{1}{\theta}\sqn{y^{k+1} - y^k}.
	\end{align*}
	Using $\beta$ definition \eqref{scary:beta} we get
	\begin{align*}
	\frac{1}{\theta}\sqn{y^{k+1} - y^*}
	&\leq
	\frac{1}{\theta}\sqn{y^k - y^*}
	+
	\bg_F(x_g^k, x^*) - \frac{\nu}{2}\sqn{x_g^k - x^*}
	-
	\beta\sqn{y^{k+1} - y^*}
	\\&-
	2\<\nu^{-1}(y_g^k + z_g^k) + x^{k+1}, y^{k+1} - y^*>
	-
	\frac{1}{\theta}\sqn{y^{k+1} - y^k}.
	\end{align*}
	Using optimality condition \eqref{opt:y} we get
	\begin{align*}
	\frac{1}{\theta}\sqn{y^{k+1} - y^*}
	&\leq
	\frac{1}{\theta}\sqn{y^k - y^*}
	+
	\bg_F(x_g^k, x^*) - \frac{\nu}{2}\sqn{x_g^k - x^*}
	-
	\beta\sqn{y^{k+1} - y^*}
	\\&-
	2\nu^{-1}\<y_g^k + z_g^k - (y^* + z^*), y^{k+1} - y^*>
	-
	2\<x^{k+1} - x^*, y^{k+1} - y^*>
	-
	\frac{1}{\theta}\sqn{y^{k+1} - y^k}.
	\end{align*}
	Using \eqref{scary:eq:1} together with $\sigma_1$ definition \eqref{scary:sigma1} we get
	\begin{align*}
	\frac{1}{\theta}\sqn{y^{k+1} - y^*}
	&\leq
	\frac{1}{\theta}\sqn{y^k - y^*}
	+
	\bg_F(x_g^k, x^*) - \frac{\nu}{2}\sqn{x_g^k - x^*}
	-
	\frac{\beta}{2}\sqn{y^{k+1} - y^*}
	\\&+
	\frac{\beta(1-\sigma_2/2)}{2\sigma_2}\sqn{y_f^k - y^*}
	-
	\frac{\beta}{2\sigma_2}\sqn{y_f^{k+1} - y^*}
	-
	\frac{\beta}{4}\sqn{y_g^k - y^*}
	+
	\frac{\beta\left(\sigma_2- \sigma_1\right)}{2}\sqn{y^{k+1} - y^k}
	\\&-
	2\nu^{-1}\<y_g^k + z_g^k - (y^* + z^*), y^{k+1} - y^*>
	-
	2\<x^{k+1} - x^*, y^{k+1} - y^*>
	-
	\frac{1}{\theta}\sqn{y^{k+1} - y^k}
	\\&\leq
	\frac{1}{\theta}\sqn{y^k - y^*}
	-
	\frac{\beta}{2}\sqn{y^{k+1} - y^*}
	+
	\frac{\beta(1-\sigma_2/2)}{2\sigma_2}\sqn{y_f^k - y^*}
	-
	\frac{\beta}{2\sigma_2}\sqn{y_f^{k+1} - y^*}
	\\&+
	\bg_F(x_g^k, x^*) - \frac{\nu}{2}\sqn{x_g^k - x^*}
	-
	\frac{\beta}{4}\sqn{y_g^k - y^*}
	+
	\left(\frac{\beta\sigma_2^2}{4} - \frac{1}{\theta}\right)\sqn{y^{k+1} - y^k}
	\\&-
	2\nu^{-1}\<y_g^k + z_g^k - (y^* + z^*), y^{k+1} - y^*>
	-
	2\<x^{k+1} - x^*, y^{k+1} - y^*>.
	\end{align*}
	Rearranging gives 
	\begin{align*}
	\MoveEqLeft[4]
	\left(\frac{1}{\theta} + \frac{\beta}{2}\right)\sqn{y^{k+1} - y^*}
	+
	\frac{\beta}{2\sigma_2}\sqn{y_f^{k+1} - y^*}\\
	&\leq
	\frac{1}{\theta}\sqn{y^k - y^*}
	+
	\frac{\beta(1-\sigma_2/2)}{2\sigma_2}\sqn{y_f^k - y^*}
	+
	\bg_F(x_g^k, x^*) - \frac{\nu}{2}\sqn{x_g^k - x^*}
	-
	2\<x^{k+1} - x^*, y^{k+1} - y^*>
	\\&-
	2\nu^{-1}\<y_g^k + z_g^k - (y^* + z^*), y^{k+1} - y^*>
	-
	\frac{\beta}{4}\sqn{y_g^k - y^*}
	+
	\left(\frac{\beta\sigma_2^2}{4} - \frac{1}{\theta}\right)\sqn{y^{k+1} - y^k}.
	\end{align*}
\end{proof}

\newpage

\begin{lemma}
	The following inequality holds:
	\begin{equation}\label{scary:eq:2}
	\sqn{m^k}_\mP
	\leq
	8\chi^2\gamma^2\nu^{-2}\sqn{y_g^k + z_g^k}_\mP + 4\chi(1 - (4\chi)^{-1})\sqn{m^k}_\mP - 4\chi\sqn{m^{k+1}}_\mP.
	\end{equation}
\end{lemma}
\begin{proof}
	Using Line~\ref{scary:line:m} of Algorithm~\ref{scary:alg} we get
	\begin{align*}
	\sqn{m^{k+1}}_\mP
	&=
	\sqn{\gamma\nu^{-1}(y_g^k+z_g^k) + m^k - (\mW(k)\otimes \mI_d)\left[\gamma\nu^{-1}(y_g^k+z_g^k) + m^k\right]}_{\mP}\\
	&=
	\sqn{\mP\left[\gamma\nu^{-1}(y_g^k+z_g^k) + m^k\right]- (\mW(k)\otimes \mI_d)\mP\left[\gamma\nu^{-1}(y_g^k+z_g^k) + m^k\right]}.
	\end{align*}
	Using  property \eqref{eq:chi} we obtain
	\begin{align*}
	\sqn{m^{k+1}}_\mP
	&\leq (1 - \chi^{-1})
	\sqn{m^k + \gamma\nu^{-1}(y_g^k + z_g^k)}_\mP.
	\end{align*}
	Using inequality $\sqn{a+b} \leq (1+c)\sqn{a} + (1+c^{-1})\sqn{b}$ with $c = \frac{1}{2(\chi - 1)}$ we get
	\begin{align*}
	\sqn{m^{k+1}}_\mP
	&\leq (1 - \chi^{-1})
	\left[
	\left(1 + \frac{1}{2(\chi - 1)}\right)\sqn{m^k}_\mP
	+ 
	\left(1 +  2(\chi - 1)\right)\gamma^2\nu^{-2}\sqn{y_g^k + z_g^k}_\mP
	\right]
	\\&\leq
	(1 - (2\chi)^{-1})\sqn{m^k}_\mP
	+
	2\chi\gamma^2\nu^{-2}\sqn{y_g^k + z_g^k}_\mP.
	\end{align*}
	Rearranging gives
	\begin{align*}
	\sqn{m^k}_\mP
	&\leq
	8\chi^2\gamma^2\nu^{-2}\sqn{y_g^k + z_g^k}_\mP + 4\chi(1 - (4\chi)^{-1})\sqn{m^k}_\mP - 4\chi\sqn{m^{k+1}}_\mP.
	\end{align*}
\end{proof}

\begin{lemma}
	Let $\z^k$ be defined as follows:
	\begin{equation}\label{scary:zhat}
	\z^k = z^k - \mP m^k.
	\end{equation}
	Then the following inequality holds:
	\begin{equation}
	\begin{split}\label{scary:eq:z}
	\MoveEqLeft[4]\frac{1}{\gamma}\sqn{\z^{k+1} - z^*}
	+
	\frac{4}{3\gamma}\sqn{m^{k+1}}_\mP
	\leq
	\left(\frac{1}{\gamma} - \delta\right)\sqn{\z^k - z^*}
	+
	\left(1-(4\chi)^{-1} +\frac{3\gamma\delta}{2}\right)\frac{4}{3\gamma}\sqn{m^k}_\mP
	\\&-
	2\nu^{-1}\<y_g^k + z_g^k - (y^*+z^*),z^k - z^*>
	+
	\gamma\nu^{-2}\left(1 + 6\chi\right)\sqn{y_g^k+z_g^k}_\mP
	\\&+
	2\delta\sqn{z_g^k  - z^*}
	+
	\left(2\gamma\delta^2-\delta\right)\sqn{z_g^k - z^k}.
	\end{split}
	\end{equation}
\end{lemma}
\begin{proof}
	\begin{align*}
	\frac{1}{\gamma}\sqn{\z^{k+1} - z^*}
	&=
	\frac{1}{\gamma}\sqn{\z^k - z^*}
	+
	\frac{2}{\gamma}\<\z^{k+1} - \z^k,\z^k - z^*>
	+
	\frac{1}{\gamma}\sqn{\z^{k+1} - \z^k}.
	\end{align*}
	Lines~\ref{scary:line:z:2} and~\ref{scary:line:m} of Algorithm~\ref{scary:alg} together with $\z^k$ definition \eqref{scary:zhat} imply
	\begin{equation*}
	\z^{k+1} - \z^k = \gamma\delta(z_g^k - z^k) - \gamma\nu^{-1}\mP(y_g^k + z_g^k).
	\end{equation*}
	Hence,
	\begin{align*}
	\frac{1}{\gamma}\sqn{\z^{k+1} - z^*}
	&=
	\frac{1}{\gamma}\sqn{\z^k - z^*}
	+
	2\delta\<z_g^k - z^k,\z^k - z^*>
	-
	2\nu^{-1}\<\mP(y_g^k + z_g^k),\z^k - z^*>
	+
	\frac{1}{\gamma}\sqn{\z^{k+1} - \z^k}
	\\&=
	\frac{1}{\gamma}\sqn{\z^k - z^*}
	+
	\delta\sqn{z_g^k - \mP m^k - z^*} - \delta\sqn{\z^k - z^*} - \delta\sqn{z_g^k - z^k}
	\\&-
	2\nu^{-1}\<\mP(y_g^k + z_g^k),\z^k - z^*>
	+
	\gamma\sqn{\delta(z_g^k - z^k) - \nu^{-1}\mP(y_g^k+z_g^k)}
	\\&\leq
	\left(\frac{1}{\gamma} - \delta\right)\sqn{\z^k - z^*}
	+
	2\delta\sqn{z_g^k  - z^*}
	+
	2\delta\sqn{m^k}_\mP
	-
	\delta\sqn{z_g^k - z^k}
	\\&-
	2\nu^{-1}\<\mP(y_g^k + z_g^k),\z^k - z^*>
	+
	2\gamma\delta^2\sqn{z_g^k - z^k}
	+
	\gamma\sqn{\nu^{-1}\mP(y_g^k+z_g^k)}
	\\&\leq
	\left(\frac{1}{\gamma} - \delta\right)\sqn{\z^k - z^*}
	+
	2\delta\sqn{z_g^k  - z^*}
	+
	\left(2\gamma\delta^2-\delta\right)\sqn{z_g^k - z^k}
	\\&-
	2\nu^{-1}\<\mP(y_g^k + z_g^k),z^k - z^*>
	+
	\gamma\sqn{\nu^{-1}\mP(y_g^k+z_g^k)}
	+
	2\delta\sqn{m^k}_\mP
	+
	2\nu^{-1}\<\mP(y_g^k + z_g^k),m^k>.
	\end{align*}
	Using the fact that $z^k \in \cL^\perp$ for all $k=0,1,2\ldots$ and optimality condition \eqref{opt:z} we get
	\begin{align*}
	\frac{1}{\gamma}\sqn{\z^{k+1} - z^*}
	&\leq
	\left(\frac{1}{\gamma} - \delta\right)\sqn{\z^k - z^*}
	+
	2\delta\sqn{z_g^k  - z^*}
	+
	\left(2\gamma\delta^2-\delta\right)\sqn{z_g^k - z^k}
	\\&-
	2\nu^{-1}\<y_g^k + z_g^k - (y^*+z^*),z^k - z^*>
	+
	\gamma\nu^{-2}\sqn{y_g^k+z_g^k}_\mP
	\\&+
	2\delta\sqn{m^k}_\mP
	+
	2\nu^{-1}\<\mP(y_g^k + z_g^k),m^k>.
	\end{align*}
	Using Young's inequality we get
	\begin{align*}
	\frac{1}{\gamma}\sqn{\z^{k+1} - z^*}
	&\leq
	\left(\frac{1}{\gamma} - \delta\right)\sqn{\z^k - z^*}
	+
	2\delta\sqn{z_g^k  - z^*}
	+
	\left(2\gamma\delta^2-\delta\right)\sqn{z_g^k - z^k}
	\\&-
	2\nu^{-1}\<y_g^k + z_g^k - (y^*+z^*),z^k - z^*>
	+
	\gamma\nu^{-2}\sqn{y_g^k+z_g^k}_\mP
	\\&+
	2\delta\sqn{m^k}_\mP
	+
	3\gamma\chi\nu^{-2}\sqn{y_g^k + z_g^k}_\mP + \frac{1}{3\gamma\chi}\sqn{m^k}_\mP.
	\end{align*}
	Using \eqref{scary:eq:2} we get
	\begin{align*}
	\frac{1}{\gamma}\sqn{\z^{k+1} - z^*}
	&\leq
	\left(\frac{1}{\gamma} - \delta\right)\sqn{\z^k - z^*}
	+
	2\delta\sqn{z_g^k  - z^*}
	+
	\left(2\gamma\delta^2-\delta\right)\sqn{z_g^k - z^k}
	\\&-
	2\nu^{-1}\<y_g^k + z_g^k - (y^*+z^*),z^k - z^*>
	+
	\gamma\nu^{-2}\sqn{y_g^k+z_g^k}_\mP
	\\&+
	2\delta\sqn{m^k}_\mP
	+
	6\gamma\nu^{-2}\chi\sqn{y_g^k + z_g^k}_\mP + \frac{4(1 - (4\chi)^{-1})}{3\gamma}\sqn{m^k}_\mP - \frac{4}{3\gamma}\sqn{m^{k+1}}_\mP
	\\&=
	\left(\frac{1}{\gamma} - \delta\right)\sqn{\z^k - z^*}
	+
	2\delta\sqn{z_g^k  - z^*}
	+
	\left(2\gamma\delta^2-\delta\right)\sqn{z_g^k - z^k}
	\\&-
	2\nu^{-1}\<y_g^k + z_g^k - (y^*+z^*),z^k - z^*>
	+
	\gamma\nu^{-2}\left(1 + 6\chi\right)\sqn{y_g^k+z_g^k}_\mP
	\\&+
	\left(1-(4\chi)^{-1}+\frac{3\gamma\delta}{2}\right)\frac{4}{3\gamma}\sqn{m^k}_\mP - \frac{4}{3\gamma}\sqn{m^{k+1}}_\mP.
	\end{align*}
\end{proof}

\begin{lemma}
	The following inequality holds:
	\begin{equation}
	\begin{split}\label{scary:eq:3}
	\MoveEqLeft[3]2\<y_g^k + z_g^k - (y^*+z^*),y^k + z^k - (y^*+ z^*)>
	\\&\geq
	2\sqn{y_g^k + z_g^k - (y^*+z^*)}
	+
	\frac{(1-\sigma_2/2)}{\sigma_2}\left(\sqn{y_g^k + z_g^k - (y^*+z^*)}  - \sqn{y_f^k + z_f^k - (y^*+z^*)}\right).
	\end{split}
	\end{equation}
\end{lemma}
\begin{proof}
	\begin{align*}
	\MoveEqLeft[4]2\<y_g^k + z_g^k - (y^*+z^*),y^k + z^k - (y^*+ z^*)>
	\\&=
	2\sqn{y_g^k + z_g^k - (y^*+z^*)}
	+
	2\<y_g^k + z_g^k - (y^*+z^*),y^k + z^k - (y_g^k + z_g^k)>.
	\end{align*}
	Using Lines~\ref{scary:line:y:1} and~\ref{scary:line:z:1} of Algorithm~\ref{scary:alg} we get
	\begin{align*}
	\MoveEqLeft[4]2\<y_g^k + z_g^k - (y^*+z^*),y^k + z^k - (y^*+ z^*)>
	\\&=
	2\sqn{y_g^k + z_g^k - (y^*+z^*)}
	+
	\frac{2(1-\sigma_1)}{\sigma_1}\<y_g^k + z_g^k - (y^*+z^*), y_g^k + z_g^k - (y_f^k + z_f^k)>
	\\&=
	2\sqn{y_g^k + z_g^k - (y^*+z^*)}
	\\&+
	\frac{(1-\sigma_1)}{\sigma_1}\left(\sqn{y_g^k + z_g^k - (y^*+z^*)} + \sqn{y_g^k + z_g^k - (y_f^k + z_f^k)} - \sqn{y_f^k + z_f^k - (y^*+z^*)}\right)
	\\&\geq
	2\sqn{y_g^k + z_g^k - (y^*+z^*)}
	+
	\frac{(1-\sigma_1)}{\sigma_1}\left(\sqn{y_g^k + z_g^k - (y^*+z^*)}  - \sqn{y_f^k + z_f^k - (y^*+z^*)}\right).
	\end{align*}
	Using $\sigma_1$ definition \eqref{scary:sigma1} we get
	\begin{align*}
	\MoveEqLeft[4]2\<y_g^k + z_g^k - (y^*+z^*),y^k + z^k - (y^*+ z^*)>
	\\&\geq
	2\sqn{y_g^k + z_g^k - (y^*+z^*)}
	+
	\frac{(1-\sigma_2/2)}{\sigma_2}\left(\sqn{y_g^k + z_g^k - (y^*+z^*)}  - \sqn{y_f^k + z_f^k - (y^*+z^*)}\right).
	\end{align*}
\end{proof}

\begin{lemma}
	Let $\zeta$ be defined by
	\begin{equation}\label{scary:zeta}
	\zeta = 1/2.
	\end{equation}
	Then the following inequality holds:
	\begin{equation}
	\begin{split}\label{scary:eq:4}
	\MoveEqLeft[6]-2\<y^{k+1} - y^k,y_g^k + z_g^k - (y^*+z^*)>
	\\&\leq
	\frac{1}{\sigma_2}\sqn{y_g^k + z_g^k - (y^*+z^*)}
	-
	\frac{1}{\sigma_2}\sqn{y_f^{k+1} + z_f^{k+1} - (y^*+z^*)}
	\\&+
	2\sigma_2\sqn{y^{k+1} - y^k}
	-
	\frac{1}{2\sigma_2\chi}\sqn{y_g^k + z_g^k}_{\mP}.
	\end{split}
	\end{equation}
\end{lemma}
\begin{proof}
	\begin{align*}
	\MoveEqLeft[4]\sqn{y_f^{k+1} + z_f^{k+1} - (y^*+z^*)}
	\\&=
	\sqn{y_g^k + z_g^k - (y^*+z^*)} + 2\<y_f^{k+1} + z_f^{k+1} - (y_g^k + z_g^k),y_g^k + z_g^k - (y^*+z^*)>
	\\&+
	\sqn{y_f^{k+1} + z_f^{k+1} - (y_g^k + z_g^k)}
	\\&\leq
	\sqn{y_g^k + z_g^k - (y^*+z^*)} + 2\<y_f^{k+1} + z_f^{k+1} - (y_g^k + z_g^k),y_g^k + z_g^k - (y^*+z^*)>
	\\&+
	2\sqn{y_f^{k+1} - y_g^k}
	+
	2\sqn{z_f^{k+1} - z_g^k}.
	\end{align*}
	Using Line~\ref{scary:line:y:3} of Algorithm~\ref{scary:alg} we get
	\begin{align*}
	\MoveEqLeft[4]\sqn{y_f^{k+1} + z_f^{k+1} - (y^*+z^*)}
	\\&\leq
	\sqn{y_g^k + z_g^k - (y^*+z^*)}
	+
	2\sigma_2\<y^{k+1} - y^k,y_g^k + z_g^k - (y^*+z^*)>
	+
	2\sigma_2^2\sqn{y^{k+1} - y^k}
	\\&+
	2\<z_f^{k+1} - z_g^k,y_g^k + z_g^k - (y^*+z^*)>
	+
	2\sqn{z_f^{k+1} - z_g^k}.
	\end{align*}
	Using Line~\ref{scary:line:z:3} of Algorithm~\ref{scary:alg} and optimality condition \eqref{opt:z} we get
	\begin{align*}
	\MoveEqLeft[4]\sqn{y_f^{k+1} + z_f^{k+1} - (y^*+z^*)}
	\\&\leq
	\sqn{y_g^k + z_g^k - (y^*+z^*)}
	+
	2\sigma_2\<y^{k+1} - y^k,y_g^k + z_g^k - (y^*+z^*)>
	+
	2\sigma_2^2\sqn{y^{k+1} - y^k}
	\\&-
	2\zeta\< (\mW(k)\otimes \mI_d)(y_g^k + z_g^k),y_g^k + z_g^k - (y^*+z^*)>
	+
	2\zeta^2\sqn{(\mW(k)\otimes \mI_d)(y_g^k + z_g^k)}
	\\&=
	\sqn{y_g^k + z_g^k - (y^*+z^*)}
	+
	2\sigma_2\<y^{k+1} - y^k,y_g^k + z_g^k - (y^*+z^*)>
	+
	2\sigma_2^2\sqn{y^{k+1} - y^k}
	\\&-
	2\zeta\< (\mW(k)\otimes \mI_d)(y_g^k + z_g^k),y_g^k + z_g^k>
	+
	2\zeta^2\sqn{(\mW(k)\otimes \mI_d)(y_g^k + z_g^k)}.
	\end{align*}
	Using $\zeta$ definition \eqref{scary:zeta} we get
	\begin{align*}
	\MoveEqLeft[4]\sqn{y_f^{k+1} + z_f^{k+1} - (y^*+z^*)}
	\\&\leq
	\sqn{y_g^k + z_g^k - (y^*+z^*)}
	+
	2\sigma_2\<y^{k+1} - y^k,y_g^k + z_g^k - (y^*+z^*)>
	+
	2\sigma_2^2\sqn{y^{k+1} - y^k}
	\\&-
	\< (\mW(k)\otimes \mI_d)(y_g^k + z_g^k),y_g^k + z_g^k>
	+
	\frac{1}{2}\sqn{(\mW(k)\otimes \mI_d)(y_g^k + z_g^k)}
	\\&=
	\sqn{y_g^k + z_g^k - (y^*+z^*)}
	+
	2\sigma_2\<y^{k+1} - y^k,y_g^k + z_g^k - (y^*+z^*)>
	+
	2\sigma_2^2\sqn{y^{k+1} - y^k}
	\\&-
	\frac{1}{2}\sqn{(\mW(k)\otimes \mI_d)(y_g^k + z_g^k)}
	-
	\frac{1}{2}\sqn{y_g^k + z_g^k}
	+
	\frac{1}{2}\sqn{(\mW(k)\otimes \mI_d)(y_g^k + z_g^k) - (y_g^k + z_g^k)}
	\\&+
	\frac{1}{2}\sqn{(\mW(k)\otimes \mI_d)(y_g^k + z_g^k)}
	\\&\leq
	\sqn{y_g^k + z_g^k - (y^*+z^*)}
	+
	2\sigma_2\<y^{k+1} - y^k,y_g^k + z_g^k - (y^*+z^*)>
	+
	2\sigma_2^2\sqn{y^{k+1} - y^k}
	\\&-
	\frac{1}{2}\sqn{y_g^k + z_g^k}_\mP
	+
	\frac{1}{2}\sqn{(\mW(k)\otimes \mI_d)(y_g^k + z_g^k) - (y_g^k + z_g^k)}_\mP.
	\\&=
	\sqn{y_g^k + z_g^k - (y^*+z^*)}
	+
	2\sigma_2\<y^{k+1} - y^k,y_g^k + z_g^k - (y^*+z^*)>
	+
	2\sigma_2^2\sqn{y^{k+1} - y^k}
	\\&-
	\frac{1}{2}\sqn{y_g^k + z_g^k}_\mP
	+
	\frac{1}{2}\sqn{(\mW(k)\otimes \mI_d)\mP(y_g^k + z_g^k) - \mP(y_g^k + z_g^k)}.
	\end{align*}
	Using condition \eqref{eq:chi} we get
	\begin{align*}
	\MoveEqLeft[4]\sqn{y_f^{k+1} + z_f^{k+1} - (y^*+z^*)}
	\\&\leq
	\sqn{y_g^k + z_g^k - (y^*+z^*)}
	+
	2\sigma_2\<y^{k+1} - y^k,y_g^k + z_g^k - (y^*+z^*)>
	+
	2\sigma_2^2\sqn{y^{k+1} - y^k}
	\\&-
	(2\chi)^{-1}\sqn{y_g^k + z_g^k}_\mP.
	\end{align*}
	Rearranging gives
	\begin{align*}
	\MoveEqLeft[6]-2\<y^{k+1} - y^k,y_g^k + z_g^k - (y^*+z^*)>
	\\&\leq
	\frac{1}{\sigma_2}\sqn{y_g^k + z_g^k - (y^*+z^*)}
	-
	\frac{1}{\sigma_2}\sqn{y_f^{k+1} + z_f^{k+1} - (y^*+z^*)}
	\\&+
	2\sigma_2\sqn{y^{k+1} - y^k}
	-
	\frac{1}{2\sigma_2\chi}\sqn{y_g^k + z_g^k}_{\mP}.
	\end{align*}
\end{proof} 

\begin{lemma}
	Let $\delta$ be defined as follows:
	\begin{equation}\label{scary:delta}
	\delta = \frac{1}{17L}.
	\end{equation}
	Let $\gamma$ be defined as follows:
	\begin{equation}\label{scary:gamma}
	\gamma = \frac{\nu}{14\sigma_2\chi^2}.
	\end{equation}
	Let $\theta$ be defined as follows:
	\begin{equation}\label{scary:theta}
	\theta = \frac{\nu}{4\sigma_2}.
	\end{equation}
	Let $\sigma_2$ be defined as follows:
	\begin{equation}\label{scary:sigma2}
	\sigma_2 = \frac{\sqrt{\mu}}{16\chi\sqrt{L}}.
	\end{equation}
	Let $\Psi_{yz}^k$ be the following Lyapunov function
	\begin{equation}
	\begin{split}\label{scary:Psi_yz}
	\Psi_{yz}^k &= \left(\frac{1}{\theta} + \frac{\beta}{2}\right)\sqn{y^{k} - y^*}
	+
	\frac{\beta}{2\sigma_2}\sqn{y_f^{k} - y^*}
	+
	\frac{1}{\gamma}\sqn{\z^{k} - z^*}
	\\&+
	\frac{4}{3\gamma}\sqn{m^{k}}_\mP
	+
	\frac{\nu^{-1}}{\sigma_2}\sqn{y_f^{k} + z_f^{k} - (y^*+z^*)}.
	\end{split}
	\end{equation}
	Then the following inequality holds:
	\begin{equation}\label{scary:eq:yz}
	\Psi_{yz}^{k+1} \leq \left(1 - \frac{\sqrt{\mu}}{32\chi\sqrt{L}}\right)\Psi_{yz}^k
	+
	\bg_F(x_g^k, x^*) - \frac{\nu}{2}\sqn{x_g^k - x^*}
	-
	2\<x^{k+1} - x^*, y^{k+1} - y^*>.
	\end{equation}
\end{lemma}
\begin{proof}
	Combining \eqref{scary:eq:y} and \eqref{scary:eq:z} gives
	\begin{align*}
	\MoveEqLeft[4]\left(\frac{1}{\theta} + \frac{\beta}{2}\right)\sqn{y^{k+1} - y^*}
	+
	\frac{\beta}{2\sigma_2}\sqn{y_f^{k+1} - y^*}
	+
	\frac{1}{\gamma}\sqn{\z^{k+1} - z^*}
	+
	\frac{4}{3\gamma}\sqn{m^{k+1}}_\mP
	\\&\leq
	\left(\frac{1}{\gamma} - \delta\right)\sqn{\z^k - z^*}
	+
	\left(1-(4\chi)^{-1}+\frac{3\gamma\delta}{2}\right)\frac{4}{3\gamma}\sqn{m^k}_\mP
	+
	\frac{1}{\theta}\sqn{y^k - y^*}
	+
	\frac{\beta(1-\sigma_2/2)}{2\sigma_2}\sqn{y_f^k - y^*}
	\\&-
	2\nu^{-1}\<y_g^k + z_g^k - (y^*+z^*),y^k + z^k - (y^*+ z^*)>
	-
	2\nu^{-1}\<y_g^k + z_g^k - (y^* + z^*), y^{k+1} - y^k>
	\\&+
	\gamma\nu^{-2}\left(1 + 6\chi\right)\sqn{y_g^k+z_g^k}_\mP
	+
	\left(\frac{\beta\sigma_2^2}{4} - \frac{1}{\theta}\right)\sqn{y^{k+1} - y^k}
	+
	2\delta\sqn{z_g^k  - z^*}
	-
	\frac{\beta}{4}\sqn{y_g^k - y^*}
	\\&+
	\bg_F(x_g^k, x^*) - \frac{\nu}{2}\sqn{x_g^k - x^*}
	-
	2\<x^{k+1} - x^*, y^{k+1} - y^*>
	+\left(2\gamma\delta^2-\delta\right)\sqn{z_g^k - z^k}.
	\end{align*}
	Using \eqref{scary:eq:3} and \eqref{scary:eq:4} we get
	\begin{align*}
	\MoveEqLeft[4]\left(\frac{1}{\theta} + \frac{\beta}{2}\right)\sqn{y^{k+1} - y^*}
	+
	\frac{\beta}{2\sigma_2}\sqn{y_f^{k+1} - y^*}
	+
	\frac{1}{\gamma}\sqn{\z^{k+1} - z^*}
	+
	\frac{4}{3\gamma}\sqn{m^{k+1}}_\mP
	\\&\leq
	\left(\frac{1}{\gamma} - \delta\right)\sqn{\z^k - z^*}
	+
	\left(1- (4\chi)^{-1}+\frac{3\gamma\delta}{2}\right)\frac{4}{3\gamma}\sqn{m^k}_\mP
	+
	\frac{1}{\theta}\sqn{y^k - y^*}
	+
	\frac{\beta(1-\sigma_2/2)}{2\sigma_2}\sqn{y_f^k - y^*}
	\\&-
	2\nu^{-1}\sqn{y_g^k + z_g^k - (y^*+z^*)}
	+
	\frac{\nu^{-1}(1-\sigma_2/2)}{\sigma_2}\left(\sqn{y_f^k + z_f^k - (y^*+z^*)} - \sqn{y_g^k + z_g^k - (y^*+z^*)}\right)
	\\&+
	\frac{\nu^{-1}}{\sigma_2}\sqn{y_g^k + z_g^k - (y^*+z^*)}
	-
	\frac{\nu^{-1}}{\sigma_2}\sqn{y_f^{k+1} + z_f^{k+1} - (y^*+z^*)}
	+
	2\nu^{-1}\sigma_2\sqn{y^{k+1} - y^k}
	\\&-
	\frac{\nu^{-1}}{2\sigma_2\chi}\sqn{y_g^k + z_g^k}_{\mP}
	+
	\gamma\nu^{-2}\left(1 + 6\chi\right)\sqn{y_g^k+z_g^k}_\mP
	+
	\left(\frac{\beta\sigma_2^2}{4} - \frac{1}{\theta}\right)\sqn{y^{k+1} - y^k}
	+
	2\delta\sqn{z_g^k  - z^*}
	\\&-
	\frac{\beta}{4}\sqn{y_g^k - y^*}
	+
	\bg_F(x_g^k, x^*) - \frac{\nu}{2}\sqn{x_g^k - x^*}
	-
	2\<x^{k+1} - x^*, y^{k+1} - y^*>
	+\left(2\gamma\delta^2-\delta\right)\sqn{z_g^k - z^k}
	\\&=
	\left(\frac{1}{\gamma} - \delta\right)\sqn{\z^k - z^*}
	+
	\left(1-(4\chi)^{-1}+\frac{3\gamma\delta}{2}\right)\frac{4}{3\gamma}\sqn{m^k}_\mP
	+
	\frac{1}{\theta}\sqn{y^k - y^*}
	+
	\frac{\beta(1-\sigma_2/2)}{2\sigma_2}\sqn{y_f^k - y^*}
	\\&+
	\frac{\nu^{-1}(1-\sigma_2/2)}{\sigma_2}\sqn{y_f^k + z_f^k - (y^*+z^*)}
	-
	\frac{\nu^{-1}}{\sigma_2}\sqn{y_f^{k+1} + z_f^{k+1} - (y^*+z^*)}
	\\&+
	2\delta\sqn{z_g^k  - z^*}
	-
	\frac{\beta}{4}\sqn{y_g^k - y^*}
	+
	\nu^{-1}\left(\frac{1}{\sigma_2} - \frac{(1-\sigma_2/2)}{\sigma_2} - 2\right)\sqn{y_g^k + z_g^k - (y^*+z^*)}
	\\&+
	\left(\gamma\nu^{-2}\left(1 + 6\chi\right) - \frac{\nu^{-1}}{2\sigma_2\chi}\right)\sqn{y_g^k+z_g^k}_\mP
	+
	\left(\frac{\beta\sigma_2^2}{4} + 	2\nu^{-1}\sigma_2 - \frac{1}{\theta}\right)\sqn{y^{k+1} - y^k}
	\\&+
	\left(2\gamma\delta^2-\delta\right)\sqn{z_g^k - z^k}
	+
	\bg_F(x_g^k, x^*) - \frac{\nu}{2}\sqn{x_g^k - x^*}
	-
	2\<x^{k+1} - x^*, y^{k+1} - y^*>
	\\&=
	\left(\frac{1}{\gamma} - \delta\right)\sqn{\z^k - z^*}
	+
	\left(1-(4\chi)^{-1}+\frac{3\gamma\delta}{2}\right)\frac{4}{3\gamma}\sqn{m^k}_\mP
	+
	\frac{1}{\theta}\sqn{y^k - y^*}
	+
	\frac{\beta(1-\sigma_2/2)}{2\sigma_2}\sqn{y_f^k - y^*}
	\\&+
	\frac{\nu^{-1}(1-\sigma_2/2)}{\sigma_2}\sqn{y_f^k + z_f^k - (y^*+z^*)}
	-
	\frac{\nu^{-1}}{\sigma_2}\sqn{y_f^{k+1} + z_f^{k+1} - (y^*+z^*)}
	\\&+
	2\delta\sqn{z_g^k  - z^*}
	-
	\frac{\beta}{4}\sqn{y_g^k - y^*}
	-
	\frac{3\nu^{-1}}{2}\sqn{y_g^k + z_g^k - (y^*+z^*)}
	+
	\left(2\gamma\delta^2-\delta\right)\sqn{z_g^k - z^k}
	\\&+
	\left(\gamma\nu^{-2}\left(1 + 6\chi\right) - \frac{\nu^{-1}}{2\sigma_2\chi}\right)\sqn{y_g^k+z_g^k}_\mP
	+
	\left(\frac{\beta\sigma_2^2}{4} + 	2\nu^{-1}\sigma_2 - \frac{1}{\theta}\right)\sqn{y^{k+1} - y^k}
	\\&+
	\bg_F(x_g^k, x^*) - \frac{\nu}{2}\sqn{x_g^k - x^*}
	-
	2\<x^{k+1} - x^*, y^{k+1} - y^*>.
	\end{align*}
	Using $\beta$ definition \eqref{scary:beta} and $\nu$ definition \eqref{scary:nu} we get
	\begin{align*}
	\MoveEqLeft[4]\left(\frac{1}{\theta} + \frac{\beta}{2}\right)\sqn{y^{k+1} - y^*}
	+
	\frac{\beta}{2\sigma_2}\sqn{y_f^{k+1} - y^*}
	+
	\frac{1}{\gamma}\sqn{\z^{k+1} - z^*}
	+
	\frac{4}{3\gamma}\sqn{m^{k+1}}_\mP
	\\&\leq
	\left(\frac{1}{\gamma} - \delta\right)\sqn{\z^k - z^*}
	+
	\left(1-(4\chi)^{-1}+\frac{3\gamma\delta}{2}\right)\frac{4}{3\gamma}\sqn{m^k}_\mP
	+
	\frac{1}{\theta}\sqn{y^k - y^*}
	+
	\frac{\beta(1-\sigma_2/2)}{2\sigma_2}\sqn{y_f^k - y^*}
	\\&+
	\frac{\nu^{-1}(1-\sigma_2/2)}{\sigma_2}\sqn{y_f^k + z_f^k - (y^*+z^*)}
	-
	\frac{\nu^{-1}}{\sigma_2}\sqn{y_f^{k+1} + z_f^{k+1} - (y^*+z^*)}
	\\&+
	2\delta\sqn{z_g^k  - z^*}
	-
	\frac{1}{8L}\sqn{y_g^k - y^*}
	-
	\frac{3}{\mu}\sqn{y_g^k + z_g^k - (y^*+z^*)}
	+
	\left(2\gamma\delta^2-\delta\right)\sqn{z_g^k - z^k}
	\\&+
	\left(\gamma\nu^{-2}\left(1 + 6\chi\right) - \frac{\nu^{-1}}{2\sigma_2\chi}\right)\sqn{y_g^k+z_g^k}_\mP
	+
	\left(\frac{\beta\sigma_2^2}{4} + 	2\nu^{-1}\sigma_2 - \frac{1}{\theta}\right)\sqn{y^{k+1} - y^k}
	\\&+
	\bg_F(x_g^k, x^*) - \frac{\nu}{2}\sqn{x_g^k - x^*}
	-
	2\<x^{k+1} - x^*, y^{k+1} - y^*>.
	\end{align*}
	Using $\delta$ definition \eqref{scary:delta} we get
	\begin{align*}
	\MoveEqLeft[4]\left(\frac{1}{\theta} + \frac{\beta}{2}\right)\sqn{y^{k+1} - y^*}
	+
	\frac{\beta}{2\sigma_2}\sqn{y_f^{k+1} - y^*}
	+
	\frac{1}{\gamma}\sqn{\z^{k+1} - z^*}
	+
	\frac{4}{3\gamma}\sqn{m^{k+1}}_\mP
	\\&\leq
	\left(\frac{1}{\gamma} - \delta\right)\sqn{\z^k - z^*}
	+
	\left(1-(4\chi)^{-1}+\frac{3\gamma\delta}{2}\right)\frac{4}{3\gamma}\sqn{m^k}_\mP
	+
	\frac{1}{\theta}\sqn{y^k - y^*}
	+
	\frac{\beta(1-\sigma_2/2)}{2\sigma_2}\sqn{y_f^k - y^*}
	\\&+
	\frac{\nu^{-1}(1-\sigma_2/2)}{\sigma_2}\sqn{y_f^k + z_f^k - (y^*+z^*)}
	-
	\frac{\nu^{-1}}{\sigma_2}\sqn{y_f^{k+1} + z_f^{k+1} - (y^*+z^*)}
	\\&+
	\left(\gamma\nu^{-2}\left(1 + 6\chi\right) - \frac{\nu^{-1}}{2\sigma_2\chi}\right)\sqn{y_g^k+z_g^k}_\mP
	+
	\left(\frac{\beta\sigma_2^2}{4} + 	2\nu^{-1}\sigma_2 - \frac{1}{\theta}\right)\sqn{y^{k+1} - y^k}
	\\&+
	\left(2\gamma\delta^2-\delta\right)\sqn{z_g^k - z^k}
	+
	\bg_F(x_g^k, x^*) - \frac{\nu}{2}\sqn{x_g^k - x^*}
	-
	2\<x^{k+1} - x^*, y^{k+1} - y^*>.
	\end{align*}
	Using $\gamma$ definition \eqref{scary:gamma} we get
	\begin{align*}
	\MoveEqLeft[4]\left(\frac{1}{\theta} + \frac{\beta}{2}\right)\sqn{y^{k+1} - y^*}
	+
	\frac{\beta}{2\sigma_2}\sqn{y_f^{k+1} - y^*}
	+
	\frac{1}{\gamma}\sqn{\z^{k+1} - z^*}
	+
	\frac{4}{3\gamma}\sqn{m^{k+1}}_\mP
	\\&\leq
	\left(\frac{1}{\gamma} - \delta\right)\sqn{\z^k - z^*}
	+
	\left(1-(4\chi)^{-1}+\frac{3\gamma\delta}{2}\right)\frac{4}{3\gamma}\sqn{m^k}_\mP
	+
	\frac{1}{\theta}\sqn{y^k - y^*}
	+
	\frac{\beta(1-\sigma_2/2)}{2\sigma_2}\sqn{y_f^k - y^*}
	\\&+
	\frac{\nu^{-1}(1-\sigma_2/2)}{\sigma_2}\sqn{y_f^k + z_f^k - (y^*+z^*)}
	-
	\frac{\nu^{-1}}{\sigma_2}\sqn{y_f^{k+1} + z_f^{k+1} - (y^*+z^*)}
	\\&+
	\left(\frac{\beta\sigma_2^2}{4} + 	2\nu^{-1}\sigma_2 - \frac{1}{\theta}\right)\sqn{y^{k+1} - y^k}
	+
	\left(2\gamma\delta^2-\delta\right)\sqn{z_g^k - z^k}
	\\&+
	\bg_F(x_g^k, x^*) - \frac{\nu}{2}\sqn{x_g^k - x^*}
	-
	2\<x^{k+1} - x^*, y^{k+1} - y^*>.
	\end{align*}
	Using $\theta$ definition together with \eqref{scary:nu}, \eqref{scary:beta} and \eqref{scary:sigma2} gives
	\begin{align*}
	\MoveEqLeft[4]\left(\frac{1}{\theta} + \frac{\beta}{2}\right)\sqn{y^{k+1} - y^*}
	+
	\frac{\beta}{2\sigma_2}\sqn{y_f^{k+1} - y^*}
	+
	\frac{1}{\gamma}\sqn{\z^{k+1} - z^*}
	+
	\frac{4}{3\gamma}\sqn{m^{k+1}}_\mP
	\\&\leq
	\left(\frac{1}{\gamma} - \delta\right)\sqn{\z^k - z^*}
	+
	\left(1-(4\chi)^{-1}+\frac{3\gamma\delta}{2}\right)\frac{4}{3\gamma}\sqn{m^k}_\mP
	+
	\frac{1}{\theta}\sqn{y^k - y^*}
	+
	\frac{\beta(1-\sigma_2/2)}{2\sigma_2}\sqn{y_f^k - y^*}
	\\&+
	\frac{\nu^{-1}(1-\sigma_2/2)}{\sigma_2}\sqn{y_f^k + z_f^k - (y^*+z^*)}
	-
	\frac{\nu^{-1}}{\sigma_2}\sqn{y_f^{k+1} + z_f^{k+1} - (y^*+z^*)}
	\\&+
	\left(2\gamma\delta^2-\delta\right)\sqn{z_g^k - z^k}
	+
	\bg_F(x_g^k, x^*) - \frac{\nu}{2}\sqn{x_g^k - x^*}
	-
	2\<x^{k+1} - x^*, y^{k+1} - y^*>.
	\end{align*}
	Using $\gamma$ definition \eqref{scary:gamma} and $\delta$ definition \eqref{scary:delta} we get
	\begin{align*}
	\MoveEqLeft[4]\left(\frac{1}{\theta} + \frac{\beta}{2}\right)\sqn{y^{k+1} - y^*}
	+
	\frac{\beta}{2\sigma_2}\sqn{y_f^{k+1} - y^*}
	+
	\frac{1}{\gamma}\sqn{\z^{k+1} - z^*}
	+
	\frac{4}{3\gamma}\sqn{m^{k+1}}_\mP
	\\&\leq
	\left(\frac{1}{\gamma} - \delta\right)\sqn{\z^k - z^*}
	+
	\left(1-(8\chi)^{-1}\right)\frac{4}{3\gamma}\sqn{m^k}_\mP
	+
	\frac{1}{\theta}\sqn{y^k - y^*}
	+
	\frac{\beta(1-\sigma_2/2)}{2\sigma_2}\sqn{y_f^k - y^*}
	\\&+
	\frac{\nu^{-1}(1-\sigma_2/2)}{\sigma_2}\sqn{y_f^k + z_f^k - (y^*+z^*)}
	-
	\frac{\nu^{-1}}{\sigma_2}\sqn{y_f^{k+1} + z_f^{k+1} - (y^*+z^*)}
	\\&+
	\bg_F(x_g^k, x^*) - \frac{\nu}{2}\sqn{x_g^k - x^*}
	-
	2\<x^{k+1} - x^*, y^{k+1} - y^*>.
	\end{align*}
	After rearranging and using $\Psi_{yz}^k$ definition \eqref{scary:Psi_yz} we get
	\begin{align*}
	\Psi_{yz}^{k+1} 
	&\leq
	\max\left\{(1 + \theta\beta/2)^{-1}, (1-\gamma\delta), (1-\sigma_2/2), (1-(8\chi)^{-1})\right\}\Psi_{yz}^k
	\\&+
	\bg_F(x_g^k, x^*) - \frac{\nu}{2}\sqn{x_g^k - x^*}
	-
	2\<x^{k+1} - x^*, y^{k+1} - y^*>
	\\&\leq
	\left(1 - \frac{\sqrt{\mu}}{32\chi\sqrt{L}}\right)\Psi_{yz}^k
	+
	\bg_F(x_g^k, x^*) - \frac{\nu}{2}\sqn{x_g^k - x^*}
	-
	2\<x^{k+1} - x^*, y^{k+1} - y^*>.
	\end{align*}
\end{proof}

\newpage
\begin{proof}[Proof of Theorem~\ref{thm:adom+_conv}]
	Combining \eqref{scary:eq:x} and \eqref{scary:eq:yz} gives
	\begin{align*}
	\Psi_x^{k+1} + \Psi_{yz}^{k+1}
	&\leq
	\left(1 - \frac{\sqrt{\mu}}{\sqrt{\mu}+2\sqrt{L}}\right)\Psi_x^k
	+
	\left(1 - \frac{\lmin\sqrt{\mu}}{32\lmax\sqrt{L}}\right)\Psi_{yz}^k
	\\&\leq
	\left(1 - \frac{\lmin\sqrt{\mu}}{32\lmax\sqrt{L}}\right)(\Psi_x^k + \Psi_{yz}^k).
	\end{align*}
	This implies
	\begin{align*}
	\Psi_x^k + \Psi_{yz}^k \leq\left(1 - \frac{\lmin\sqrt{\mu}}{32\lmax\sqrt{L}}\right)^k(\Psi_x^0 + \Psi_{yz}^0).
	\end{align*}
	Using $\Psi_x^k$ definition \eqref{scary:Psi_x} we get
	\begin{align*}
	\sqn{x^k - x^*} &\leq \eta\Psi_x^k \leq \eta (\Psi_x^k + \Psi_{yz}^k ) \leq \left(1 - \frac{\lmin\sqrt{\mu}}{32\lmax\sqrt{L}}\right)^k\eta(\Psi_x^0 + \Psi_{yz}^0).
	\end{align*}
	Choosing $C = \eta(\Psi_x^0 + \Psi_{yz}^0)$ and using the number of iterations
	\begin{equation*}
	k = 32\chi\sqrt{\nicefrac{L}{\mu}}\log \frac{C}{\varepsilon} = \cO \left( \chi\sqrt{\nicefrac{L}{\mu}}\log \frac{1}{\epsilon}\right).
	\end{equation*}
	we get
	\begin{equation*}
	\sqn{x^k - x^*} \leq \epsilon,
	\end{equation*}
	which concludes the proof.
\end{proof}